\documentclass[11pt]{article}

\usepackage{nicefrac}
\usepackage{refcount}
\usepackage{amsfonts}
\usepackage{hyperref}
\usepackage{epsfig}
\usepackage{graphicx}
\usepackage{amsmath}
\usepackage{amssymb}
\usepackage{fullpage}
\usepackage{enumitem}
\usepackage{xcolor}
\usepackage{subcaption}

\newtheorem{theorem}{Theorem}[section]
\newtheorem{definition}[theorem]{Definition}
\newtheorem{proposition}[theorem]{Proposition}

\newtheorem{claim}[theorem]{Claim}
\newtheorem{corollary}[theorem]{Corollary}
\newtheorem{remark}[theorem]{Remark}

\newtheorem{observation}[theorem]{Observation}

\renewcommand{\Re}{\mathbb{R}}

\newenvironment{proof}[1][Proof]{ \noindent \textbf{#1: }}{$\Box$
	\bigskip}

\renewcommand{\Re}{\mathbb{R}}

\newcommand{\F}{\mathcal{F}}

\newcommand{\C}{\mathcal{C}}
\newcommand{\D}{\mathcal{D}}

\title{Complements of Finite Unions of Convex Sets}
\author{Chaya Keller\thanks{Department of Computer Science, Ariel University, Israel. \texttt{chayak@ariel.ac.il}. Research partially supported by the Israel Science Foundation (grant no. 1065/20).}
	\mbox{ }
	and Micha A. Perles\thanks{Einstein Institute of Mathematics, Hebrew University, Jerusalem, Israel.
		\texttt{micha.perles@mail.huji.ac.il}}
}

\begin{document}

\maketitle

\begin{abstract}
	Finite unions of convex sets are a central object of study in discrete and computational geometry. In this paper we initiate a systematic study of \emph{complements} of such unions -- i.e., sets of the form  $S=\mathbb{R}^d \setminus (\cup_{i=1}^n K_i)$, where $K_i$ are convex sets. In the first part of the paper we study isolated points in $S$, whose number is related to the Betti numbers of $\cup_{i=1}^n K_i$ and to its non-convexity properties.
	We obtain upper bounds on the number of such points, which are sharp for $n=3$ and significantly improve previous bounds of Lawrence and Morris (2009) for all $n \ll \frac{2^d}{d}$. In the second part of the paper we study coverings of $S$ by well-behaved sets. We show that $S$ can be covered by at most $g(d,n)$ flats of different dimensions, in such a way that each $x \in S$ is covered by a flat whose dimension equals the `local dimension' of $S$ in the neighborhood of $x$. Furthermore, we determine the structure of a minimum cover that satisfies this property. Then, we study quantitative aspects of this minimum cover and obtain sharp upper bounds on its size in various settings. 
\end{abstract}

\section{Introduction}

The complexity of finite unions of convex sets has been a prolific research area in the last decades, due to intrinsic deep discrete-geometric questions pertaining to it and to applications to optimization, robotics, motion planning and other areas (see, e.g., the survey~\cite{Agarwal2008a}). In some of these questions and applications, the structure of the complement plays an important role, and as a result, various structural questions regarding complements of such finite unions were studied in different works over the years.   

One well-studied question is determining the maximum possible number of \emph{bounded connected components} in $S=\mathbb{R}^d \setminus (\cup_{i=1}^n K_i)$ (also known as \emph{voids} in $K=\cup_{i=1}^n K_i$), where $\{K_i\}$ are convex.  
Asked by Fejes T\'{o}th in the plane, and independently by Vitushkin (1958) in $\mathbb{R}^d$, this question was studied both in full generality and for specific families of convex sets, such as translates of the same convex set (see, e.g.,~\cite{AronovCDG17,Kat77}). 
As the number of voids is equal to the $(d-1)$'st Betti number of $\cup_{i=1}^n K_i$, bounding its maximum size is the first step toward understanding the statistical behavior of the Betti numbers of such unions (see~\cite{EdelsbrunnerP24}).

A related question is determining the maximum possible number of \emph{all} connected components (including unbounded ones). Kovalev~\cite{Kov88} proved that the maximum is $\sum_{i=0}^d \binom{n}{i}$, and that it is attained if and only if $K_i$ are either hyperplanes in general position, or layers between parallel hyperplanes. Bei, Chen and Zhang~\cite{BeiCZ15} provided a different proof, and applied the result to analyze the complexity of an algorithm they proposed for solving linear programming problems with only partial knowledge of the constraints.

In the case where $\{K_i\}$ are polyhedra, the combinatorial complexity of $S$ plays an important role in motion planning, and was studied, e.g., by Aronov and Sharir in~\cite{AronovS97SICOMP,AronovS97CGTA}.  

The case where $S$ is finite was studied by Lawrence and Morris~\cite{LawrenceM09}, who obtained upper bounds on $|S|$ (which, in this case, is equal to the number of \emph{one-point holes} in $K$) in terms of $n$. As was shown much earlier by Matou\v{s}ek and Valtr~\cite[Thm.~1.1(ii)]{MatousekV99}, bounds on the number of one-point holes allow bounding the \emph{convexity number} of $K$ in terms of its \emph{invisibility number} (see also~\cite[Theorem 1]{CibulkaKKMRV17}).

\medskip

In this paper we initiate a systematic study of the structural properties of such sets $S$. We concentrate on two types of questions -- isolated points in $S$ (which were studied in the special case where $S$ is finite by Lawrence and Morris~\cite{LawrenceM09}), and covering of $S$ by flats (i.e., affine subspaces of $\Re^d$).

\paragraph{Isolated points in $S$.}

For the sake of convenience, we use the following definition:
\begin{definition}
	A point $p \in \Re^d$ is \emph{encapsulated} by the convex sets $K_1, \ldots, K_n \subset \Re^d$ if for some $\epsilon>0$, $B(p,\epsilon) \setminus (\cup_{i=1}^n K_i) = \{p\}$, where $B(p,\epsilon)$ is the open ball of radius $\epsilon$ centered at $p$. In other words,  $p \in \Re^d$ is \emph{encapsulated} by $K_1, \ldots, K_n \subset \Re^d$ if these convex sets cover a pointed neighborhood of $p$.  
\end{definition}
Lawrence and Morris~\cite{LawrenceM09} studied the case $|S|< \infty$, in which every $p \in S$ is encapsulated by the convex sets $K_1,\ldots,K_n$. In the case where all $\{K_i\}$ are open, they used a classical theorem of Bj\"{o}rner and Kalai~\cite{BjornerK88} to show that $ |S| \leq \binom{n-1}{d}$. On the other hand they showed the existence of such a set $S$ with $(\lfloor \frac{n}{d} \rfloor -1)^{d} \leq |S| $. With no additional assumption on $\{K_i\}$ (except for convexity), but assuming that $S$ affinely spans $\Re^d$, they obtained the upper bound $|S| \leq (n-1) (\binom{n}{\lfloor \frac{n}{2} \rfloor})^{d-1}$.   

We consider the general case, where no additional assumptions on $S$ and $\{K_i\}$ are made.
First, we obtain the following tight bound on the number of points encapsulated by 3 convex sets:
\begin{theorem}\label{thm:3encapsuled}
	Let $K_1,K_2,K_3 \subset \Re^d$ be 3 convex sets. Denote by $S$ the set of points of $\Re^d$ encapsulated by $K_1 \cup K_2 \cup K_3 $. Define $f(d)= \lfloor \frac{3d}{2} \rfloor+1$. Then:
	
	(a) $|S| \leq f(d)$; 
	
	(b) If $|S| = f(d)$ then the sets $K_1,K_2,K_3$ are pairwise disjoint;
	
	(c) The sets $K_1,K_2,K_3$ can be chosen in such a way that $|S|=f(d)$ and $\Re^d \setminus (K_1 \cup K_2 \cup K_3) =S$. 
\end{theorem} 

For larger numbers of convex sets, let $f(d,n)$ 4fbe the maximal number of points that can be encapsulated by $n$ convex sets in $\Re^d$. It is easy to see\footnote{See also Theorem \ref{thm:two_sets-intro} below.} that $f(d,1)=0,f(d,2)=1$, for all $n \geq 1$ we have $f(0,n)=0, f(1,n)=n-1$ and by Theorem \ref{thm:3encapsuled}, $f(d,3)= \lfloor \frac{3d}{2} \rfloor+1$. The following recursive bound can be proved inductuvely.
\begin{theorem}\label{thm:recursive}
	For every $n >3,d \geq 2$,
	$$f(d,n) \leq   \sum_{i=2}^{n-1} \left(\binom{n}{i} f(d,i) \right) +f(d-2,n),$$
	and consequently,
	$$f(d,n) \leq 2d^{n-2}n!.$$	
\end{theorem} 
Note that for a constant $n$, the bound we obtain is polynomial in $d$, whereas the bound of~\cite{LawrenceM09} is exponential in $d$. Furthermore, a direct calculation shows that our bound is superior as long as $n \ll \frac{2^d}{d}$. 

In addition, we prove a sharp bound in the plane, under the additional assumptions that $S$ is finite and the sets $K_i$ are pairwise disjoint. It turns out that while in $\Re^1$ and for three sets in $\Re^2$, the maximal size of $S$ is obtained where the convex sets are pairwise disjoint, in the general case the disjointness assumption leads to a much smaller bound on $|S|$ -- even in the plane, where the upper bound is linear in $n$ (compared to a quadratic lower bound without this restriction, see Appendix~\ref{subsec:n-encapsuledR^2}).
\begin{proposition}\label{thm:numFlatsR2}
	Let $S=\mathbb{R}^2 \setminus (\cup_{i=1}^n K_i)$ where $n \geq 3$ and $\{K_i\}$ are pairwise disjoint convex sets. Assume that $|S|<\infty$. Then $|S| \leq 5n-11$, and this bound is sharp.
\end{proposition}

\paragraph{Covering $S$ by flats.} For the sake of convenience, we use the following definition.

\begin{definition}\label{def:dm}
	For $S \subset \Re^d$, the \emph{flat-dimension} of $S$, $dm(S)$, is $$dm(S)=\mbox{max}\{k: S \mbox{ includes a } k \mbox{-simplex}\},$$
	where a $k$-simplex (for $k \geq -1$) is the convex hull of $k+1$ affinely independent points in $\Re^d$. 
	
	\medskip \noindent The \emph{local dimension} $dm(S,p)$ of $S$ at $p \in S$  is $$dm(S,p)=\min_{\epsilon>0}dm(S \cap B(p,\epsilon)).$$
\end{definition}
Our main result in this part is the following theorem:
\begin{theorem}\label{thm:local}
	For any $n,d \in \mathbb{N}$, there exists a number $g=g(n,d)$ such that the following holds. Let $S=\mathbb{R}^d \setminus (\cup_{i=1}^n K_i)$, where $\{K_i\}$ are convex sets. Then $S$ can be covered by at most $g$ flats, in such a way that each $p \in S$ is covered by a flat of dimension $dm(S,p)$. Furthermore, there exists a unique such cover $\mathcal{C}$ that is minimal with respect to inclusion.  
\end{theorem} 
Actually, the proof of Theorem~\ref{thm:local} provides significant structural information on $\mathcal{C}$: 
Suppose $p \in S, dm(S,p)=k$, and let $L$ be a flat. If $S \cap W = L \cap W $ for some neighborhood $W$ of $p$, then $L$ is the $k$-flat in $\mathcal{C}$ that covers $p$.

On the quantitative side, the bound on $g(d,n)$ which follows from the proof is rather large, and in particular, the maximum numbers of flats of each dimension in $\mathcal{C}$ seem difficult to compute in general. Hence, we focus on special cases where effective bounds can be obtained. The following natural definition 
%of \emph{admissible vectors of flat dimensions} 
will be convenient. 
\begin{definition}
	%For $d,n \in \mathbb{N}$, let $V(d,n) \subset \Re^{d+1}$ be the set of vectors $(v_0,\ldots,v_{d})$ such that there exist convex sets $K_1,\ldots,K_n \subset \Re^d$ for which in the cover $\mathcal{C}$ of $S=\Re^d \setminus (\cup_{i=1}^n K_i)$, for each $0 \leq k \leq d$, the number of $k$-flats is $v_k$.
	For $S=\mathbb{R}^d \setminus (\cup_{i=1}^n K_i)$, where $\{K_i\}$ are convex sets, and for $k=0,1,\ldots,d$, denote by $\nu_k(S)$ the number of $k$-flats in the cover $\C$ of $S$. Denote by $\nu_k(d,n)$ the maximum of $\nu_k(S)$ over all such sets $S$, where $d,n,k$ are fixed.
\end{definition}
Theorems \ref{thm:3encapsuled} and \ref{thm:recursive} above provide upper bounds on $\nu_0(d,n)$ (i.e., on the maximal number of $0$-dimensional flats in $\mathcal{C}$), as by definition, the flat which covers each isolated point in $S$ must be the point itself. At the other end of the spectrum, it is clear that $\nu_d(d,n)=1$. Regarding $(d-1)$-flats, we determine $\nu_{d-1}(d,n)$ completely. 
\begin{theorem}\label{thm:num_flats-intro}
	We have 
	$$\nu_1(2,n) = t(n,4)=(\tfrac{3}{4}+o(1))\binom{n}{2}, \quad \mbox{and} \quad \nu_{d-1}(d,n) = \binom{n}{2}, \quad \forall d \geq 3,$$ where  $t(n,4)$ is the Tur{\'{a}}n number.
\end{theorem}  
Finally, in the case $n=2$, i.e., where there are only two convex sets, we completely determine the numbers of $k$-flats that can appear in the cover $\mathcal{C}$ of $S=\Re^d \setminus (K_1 \cup K_2)$, for each $0 \leq k \leq d$.
%the set $V(d,2)$ of admissible vectors.
\begin{theorem}\label{thm:two_sets-intro}
	%For $d,n \in \mathbb{N}$, let $V(d,n) \subset \Re^{d+1}$ be the set of vectors $(v_0,\ldots,v_{d})$ such that there exist convex sets $K_1,\ldots,K_n \subset \Re^d$ for which in the cover $\mathcal{C}$ of $S=\Re^d \setminus (\cup_{i=1}^n K_i)$, for each $0 \leq k \leq d$, the number of $k$-flats is $v_k$.	
%	Then for any $d$ and $n=2$, we have 	$$V(d,2)=\{0,1\}^{d+1}.$$
Let $K_1,K_2$ be convex sets in $\Re^d$ ($d \geq 1$), and let $\C$ be the cover of $S=\Re^d \setminus(K_1 \cup K_2)$ discussed above. Then any two flats $I,J \in \C$ satisfy $J \subset I $ or $I \subset J$. Consequently, $\C$ contains at most one $k$-flat for each $0 \leq k \leq d$.

Moreover, for any subset $T$ of $\{0,1,\ldots,d\}$ we can find $K_1,K_2 \subset \Re^d$ such that $\C$ contains a $k$-flat if and only if $k \in T$. 
\end{theorem}  

The rest of the paper is organized as follows. 
%A few definitions and notations are introduced in Section~\ref{sec:preliminaries}. 
In Section~\ref{sec:points} we study isolated points in $S$, proving Theorems~\ref{thm:3encapsuled} and~\ref{thm:recursive}. In Section~\ref{sec:covering} we study the qualitative question of covering $S$ by flats, proving Theorem~\ref{thm:local}. In Sections~\ref{sec:hyperplanes} and~\ref{sec:two} we study quantitative aspects of covering $S$ by flats and present the proofs of Theorems~\ref{thm:num_flats-intro} and~\ref{thm:two_sets-intro}. Finally, in Appendix~\ref{subsec:n-encapsuledR^2} we prove Proposition~\ref{thm:numFlatsR2}.

\section{Bounding the number of points encapsulated by convex sets}\label{sec:points}

Recall that $f(d,n)$ is the maximal number of points that can be encapsulated by $n$ convex sets in $\Re^d$. It is easy to see that for every $n>0$, $f(1,n)=n-1$ and that for every $d>0$, $f(d,2)=1$. After introducing a few definitions and an observation in Section~\ref{sec:preliminaries}, we determine $f(d,3)$ exactly in Section \ref{subsec:3encapsuled}. Then, in Section \ref{subsec:n-encapsulated} we obtain by an inductive argument an upper bound on $f(d,n)$ for $n>3$. 

\subsection{Preliminaries}\label{sec:preliminaries}

For a set $S \subset \Re^d$, denote by $\mbox{cl}(S)$ and $\mbox{int}(S)$ the topological closure and interior of $S$, respectively. The convex hull of $S$ is denoted by $\mbox{conv}(S)$. For $d \in \mathbb{N}$ and $0 \leq k \leq d$, a $k$-flat $\pi \subset \Re^d$ is a $k$-dimensional affine subspace of $\Re^d$.

\begin{definition}\label{def:touch}
	For $K \subset \Re^d, p \in \Re^d$, we say that $p$ touches $K$ (or $K$ touches $p$), if $p \in \mbox{cl}(K) \setminus K$.
\end{definition}

The following observation will be used several times in the sequel.
\begin{observation}\label{obs:cone}
	Let $K_1,\ldots,K_n \subset \Re^d$ be convex sets, and let $B$ be a $d$-dimensional ball $B \subset K_1 \cap\ldots \cap K_n$. Let $p_0 \notin K_1 \cup\ldots \cup K_n$. Consider the double cone with apex $p_0$ spanned by $B$. Then the other side of this cone (the dark area in Figure \ref{fig:fig1}) is disjoint to $K_1,\ldots,K_n$. 
\end{observation}
Indeed, for each $x$ in the other side of the cone above, $p_0$ lies inside a segment that connects $x$ to a point in each $K_i$.

\begin{figure}[ht]
	\begin{center}
		\scalebox{0.5}{
			\includegraphics[width=0.8\textwidth]{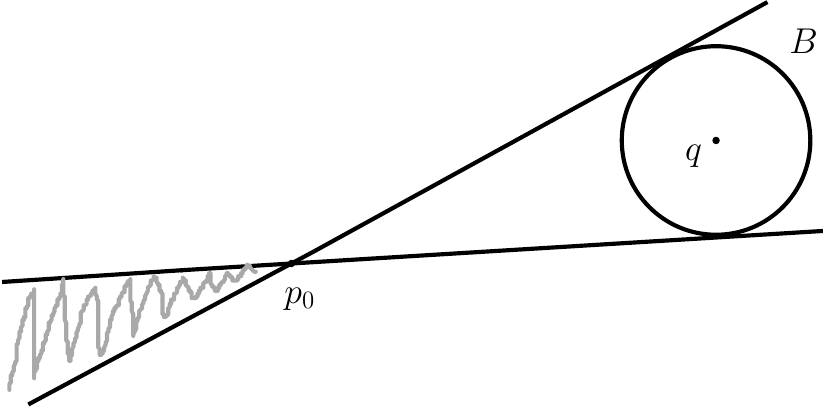}
		}
		\caption{A cone with a vertex in $p_0$ that is tangent to a ball $B$.}
		\label{fig:fig1}
	\end{center}
\end{figure}

\subsection{The number of points encapsulated by 3 convex sets in $\Re^d$}\label{subsec:3encapsuled}
The following theorem implies that $f(d,3)=\lfloor \frac{3d}{2} \rfloor+1$. For the sake of simplicity, we denote by $S$ the set of encapsulated points, instead of the entire set $\Re^d \setminus (\cup_{i=1}^n K_i)$. The reason for abusing notation is that in the proof of this theorem, elements of $\Re^d \setminus (\cup_{i=1}^n K_i)$ other than encapsulated points do not make any difference. Hence, we implicitly assume that there are no such elements, and thus, $\Re^d \setminus (\cup_{i=1}^n K_i)$ is equal to the set of encapsulated points.  

\medskip

\noindent \textbf{Theorem \ref{thm:3encapsuled} - restatement.}
	Let $K_1,K_2,K_3 \subset \Re^d$ be convex sets. Denote by $S$ the set of points of $\Re^d$ encapsulated by $K_1 \cup K_2 \cup K_3 $. Define $f(d)= \lfloor \frac{3d}{2} \rfloor+1$. Then:
	
	(a) $|S| \leq f(d)$; 
	
	(b) If $|S| = f(d)$ then the sets $K_1,K_2,K_3$ are pairwise disjoint;
	
	(c) There exist $K_1,K_2,K_3$ for which $|S|=f(d)$ and $\Re^d \setminus (K_1 \cup K_2 \cup K_3) =S$. 

\medskip
 
\begin{proof}[Proof of Theorem~\ref{thm:3encapsuled}]
	We prove the first two statements together by induction on $d$. The cases $d=0,1$ are trivial. For $d \geq 2 $ and for every $p \in S$, let $$touch(p) = \{i: 1 \leq i \leq 3, \mbox{ and } p \mbox{ }touches \mbox{ } K_i\}.$$ 
	We shall use the following observations.
	\begin{observation}\label{obs:enclosed2}
		If $p$ is encapsulated by $K_1,K_2,K_3$ then $|touch(p)| \geq2$, and if $touch(p) = \{i,j\}$ then $K_i \cap K_j = \emptyset$.
	\end{observation}
	Indeed, if $q \in K_i \cap K_j$ then all the points on the line $\ell(p,q)$ that are separated from $q$ by $p$, are not in $K_i \cup K_j$. Since there exist such points arbitrarily close to $p$, this contradicts the assumption on $p$.
	\begin{observation}\label{obs:2halfspaces}
		If for $p \in S$, $touch(p)= \{i,j\}$ then each of the cones $cone(p,K_i)= \{  (1-\lambda)p+\lambda x: \lambda>0, x\in K_i\},cone(p,K_j)$ is a semi-space\footnote{In the plane, this means that each of $cone(p,K_i), cone(p,K_j)$ is a half-open half-space. In general dimension, this means that there exists a unique orthonormal base $B=\{e_1,\ldots,e_d\}$ of $\Re^d$ such that with respect to $B$, $cone(p,K_i)$ is the set of all points that are lexicographically smaller than $p$, and $cone(p,K_j)$ is the set of all points that are lexicographically larger than $p$.}. Moreover, there is no other point $p' \in S$ with $touch(p') = \{i,j\}$. (See Figure \ref{fig:fig15}.)
	\end{observation}
	\begin{figure}[ht]
		\begin{center}
			\scalebox{0.5}{
				\includegraphics[width=0.8\textwidth]{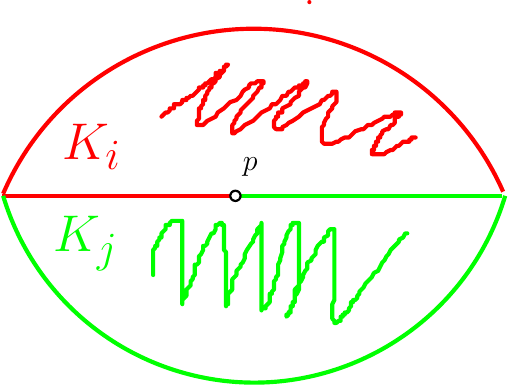}
			}
			\caption{An illustration for Observation \ref{obs:2halfspaces} in $\Re^2$ where $touch(p)=\{i,j\}$. The set $K_i$ is colored with red and $K_j$ is colored with green. In this case $cone(K_i)$ is the upper half-plane, and  $cone(K_j)$ is the lower half-plane.}
			\label{fig:fig15}
		\end{center}
	\end{figure}
	By Observation \ref{obs:2halfspaces}, there is at most one point $p_{12} \in S$ with $touch(p_{12})=\{1,2\}$, at most one point $p_{13} \in S$ with $touch(p_{13})=\{1,3\}$, and at most one point $p_{23} \in S$ with $touch(p_{23})=\{2,3\}$. By Observation \ref{obs:enclosed2}, all other points in $S$ touch all three convex sets $K_1,K_2,K_3$. Let $S' = \{p \in S: |touch(p)|=3\}$ and let $J = \mbox{aff} (S')$ be the flat spanned by $S'$. If $S' = J = \emptyset$ then we are done. Otherwise, by Observation \ref{obs:cone}, $J$ cannot contain a $d$-dimensional ball, and hence, $\mbox{dim}(J)<d$. 
	
	Consider $S \cap J$. This set includes $S'$ and maybe some of the points $p_{ij}$ defined above. Every point in $S \cap J$ is encapsulated (w.r.t.~$J$) by the three convex sets $K_i \cap J$ ($1 \leq i \leq 3$). Hence, by the induction hypothesis, $|S \cap J| \leq f(\dim (J))$. There are 3 cases: 
	
	\medskip
	
	\noindent \textbf{Case 1: $\dim (J)=d-2$:} Then $|S'| \leq f(d-2)$, and since $S$ contains at most three points that are not in $S'$ (i.e., the $p_{ij}$'s above), we have 
	$$|S| \leq |S'|+3 \leq f(d-2)+3 = f(d).$$
	If all inequalities hold with equality, then all three points $p_{12}, p_{13}$ and $p_{23}$ exist, and by Observation \ref{obs:enclosed2}, the $K_i$'s are pairwise disjoint.    
	
	\medskip \noindent \textbf{Case 2: $\dim (J)<d-2$:} Then
	$$|S| \leq |S'|+3 \leq f(\dim J)+3 < f(d).$$
	
	\noindent \textbf{Case 3: $\dim (J)=d-1$:} In this case we prove that the points $p_{12}, p_{13}$ and $p_{23}$ (if exist) are in $J$, and therefore, $|S| = |S'| \leq f(d-1) < f(d)$, and we are done.
	
	Indeed, assume to the contrary (w.l.o.g.) that $p_{12}$ exists and $p_{12} \notin J$. Then $\dim( \mbox{conv}(S' \cup \{p_{12}\})  ) =d$. Moreover, since all the points in $S'$ touch both $K_{1}$ and $K_2$, we have 
	$$\mbox{conv}(S' \cup \{p_{12}\}) \subset \mbox{cl} (K_1) \cap \mbox{cl} (K_2).$$ Therefore, $\dim (\mbox{cl} (K_1) \cap \mbox{cl} (K_2))=d$, hence $K_1 \cap K_2 \neq \emptyset$, in contradiction to Observation \ref{obs:enclosed2}. This completes the proof of parts (a) and (b) of the theorem.
	
	\medskip
	
	For the proof of (c) we need the following claim:
	\begin{claim}\label{cl:extend}
		Let $A,B \subset \Re^d$ be convex sets such that $A \cap B = \emptyset$, and let $z \in \Re^d \setminus (A \cup B)$. Then there exist two convex sets $\tilde{A},\tilde{B} \subset \Re^d$ such that 
		\begin{itemize}
			\item $A \subset \tilde{A}, B \subset \tilde{B}$,
			\item $\tilde{A} \cap \tilde{B} = \emptyset$, \mbox{and}
			\item $\tilde{A} \cup \tilde{B} = \Re^d \setminus \{z\}$,
		\end{itemize}
		if and only if 
		\begin{equation}\label{eq:1}
			A \cap (\mbox{conv}(B \cup \{z\}))= \emptyset \mbox{ and } B \cap (\mbox{conv}(A \cup \{z\}))= \emptyset.
		\end{equation}
	\end{claim} 
	\begin{proof}[Proof of Claim \ref{cl:extend}]
		The direction $(\Rightarrow)$ is trivial, since if (w.l.o.g.) $B \cap (\mbox{conv}(A \cup \{z\})) \neq \emptyset$ then there exist $a \in A$ and $b \in B$ such that $b$ in contained in the open segment $(a,z)$. But in this case, no point on the opposite ray $\{(1+\lambda)z-\lambda a: \lambda>0\}$ can belong to $\tilde{A}$ or to $\tilde{B}$, a contradiction.
		
		For the opposite direction, assume for the sake of convenience that  $z=0$ and that (\ref{eq:1}) is satisfied. An equivalent formulation of (\ref{eq:1}) is:
		\begin{equation}\label{eq:1'}
			\mbox{Each ray from 0 meets at most one of the sets }A,B.
		\end{equation}
		We shall now construct $\tilde{A}$ and $\tilde{B}$. First, let $A^+=\bigcup_{\lambda>0}(\lambda A)$ (resp., $B^+=\bigcup_{\lambda>0}(\lambda B)$) be the union of all open rays from 0 via $A$ (resp., $B$). The sets $A^+,B^+$ are convex and $A \subset A^+,B \subset B^+$. Since each ray from the origin intersects $A$ if and only if it intersects $A^+$, and similarly for $B$, the condition (\ref{eq:1'}) is satisfied also for $A^+,B^+$. Therefore, $A^+,B^+$ are disjoint convex sets that are closed under addition and multiplication by a positive scalar. Let 
		$$\D=\{F \subset \Re^d \setminus\{0\}: F \mbox{ is closed under addition and multiplication by a positive scalar}  \},$$ and consider the family
		$$\F=\{(A',B'): A',B' \in \D, A'\cap B' =\emptyset, A^+ \subset A', B^+ \subset B' \}.$$ Define a partial order on $\F$ by $(A',B') \leq (A'',B'')$ iff $A' \subseteq A''$ and $B' \subseteq B''$. Since each chain $\{(A'_i,B'_i)\}$ is bounded from above by $(\cup_i A'_i,\cup_i B'_i )$, by Zorn's lemma, $\F$ contains a maximal element $(\tilde{A},\tilde{B})$. Clearly $\tilde{A},\tilde{B}$ are disjoint convex sets, $A \subset \tilde{A}, B \subset\tilde{B}$. It remains to prove that $\tilde{A} \cup \tilde{B} = \Re^d \setminus\{0\}$. Assume on the contrary that there exists a point $w \in \Re^d \setminus(\tilde{A}\cup \tilde{B} \cup \{0\})$. Then by the maximality of $(\tilde{A},\tilde{B})$, the cone $\{\tilde{A} +\lambda w: \lambda>0\}$ meets $\tilde{B} \cup \{0\}$. Therefore there exist $a \in \tilde{A}$ and $b' \in \tilde{B} \cup \{0\}$ such that $w + a =b'$. Similarly, there exist $b \in \tilde{B}$ and $a' \in \tilde{A} \cup \{0\}$ such that $w + b =a'$. Hence $b'-a=a'-b$ and it follows that $a+a'=b+b'\in \tilde{A} \cap \tilde{B}$, a contradiction. This completes the proof of Claim \ref{cl:extend}.
	\end{proof}
	
	Claim \ref{cl:extend} implies the following consequence:
	\begin{corollary}\label{cor:extend}
		Let $H \subset \Re^d$ be a $(k+1)$-flat and let $J \subset H$ be a $k$-flat. Let $A,B \subset J$ be disjoint convex sets. Assume that $J$ separates $H$ into two open half-flats $H^+,H^-$ and that $z \in H^+$. Then there exist two disjoint convex sets $\tilde{A} , \tilde{B}$ such that $A \subset \tilde{A}, B \subset \tilde{B}$ and $\tilde{A} \cup \tilde{B}=H \setminus\{z\}$ (see Figure \ref{fig:fig16}). Moreover, the sets $\tilde{\tilde{A}}= (\tilde{A} \cap H^+) \cup A, \tilde{\tilde{B}}= (\tilde{B} \cap H^+) \cup B$ are disjoint convex sets such that $\tilde{\tilde{A}} \cap J =A,\tilde{\tilde{B}} \cap J =B$ and $(\tilde{\tilde{A}} \cup \tilde{\tilde{B}}) \cap H^+ = H^+ \setminus \{z\}$.
	\end{corollary}
	
	\begin{figure}[ht]
		\begin{center}
			\scalebox{0.5}{
				\includegraphics[width=0.8\textwidth]{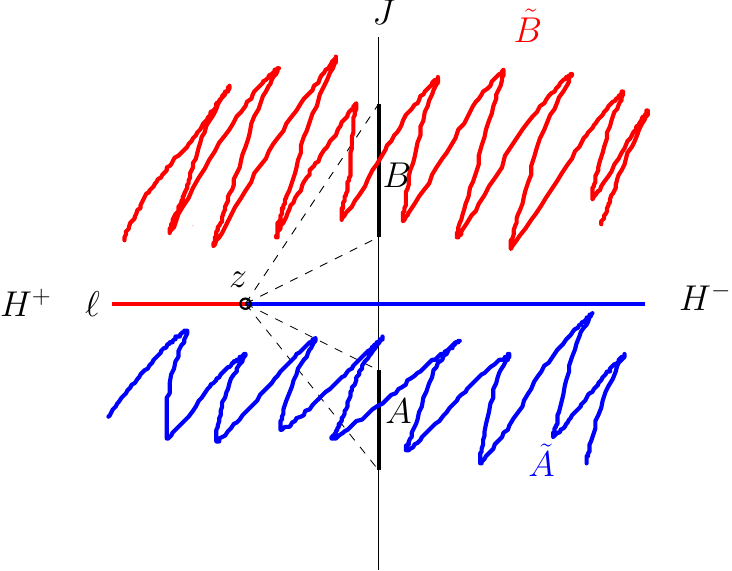}
			}
			\caption{An illustration for Corollary \ref{cor:extend} where $H = \Re^2$ and $J$ is a line. $A$ and $B$ are two segments, as illustrated in the figure. $\tilde{B}$ is the red half-open upper half-plane supported by $\ell$, and $\tilde{A}$ is the blue half-open lower half-plane supported by $\ell$.}
			\label{fig:fig16}
		\end{center}
	\end{figure}

	\medskip
	
	Now we are ready to prove part (c) of Theorem \ref{thm:3encapsuled}. We prove the construction by induction on $d$. Given a construction in $\Re^{d-2}$ of 3 convex sets that encapsulate $f(d-2)$ points, we extend each of the three convex sets to $\Re^d$ in such a way that the previously encapsulated points are still encapsulated, and three more encapsulated points are formed. Since $f(d)=f(d-2)+3$, such a construction completes the proof.  
	
	For the basic case, in $\Re^0$ we can take $K_1=K_2=K_3=\emptyset$, and in $\Re^1$ the convex sets can be $K_1=(-\infty,0), K_2=(0,1)$ and $K_3=(1,\infty)$. The construction for $d=2$ is illustrated in Figure \ref{fig:fig17} (though it can be also obtained by the inductive argument).
	
	\begin{figure}[ht]
		\begin{center}
			\scalebox{0.3}{
				\includegraphics[width=0.8\textwidth]{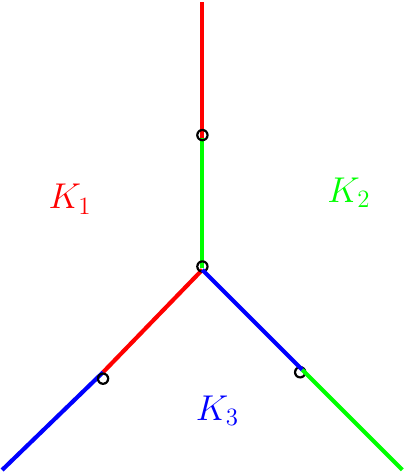}
			}
			\caption{An illustration for the basic case $d=2$ in the proof of Theorem \ref{thm:3encapsuled}(c).}
			\label{fig:fig17}
		\end{center}
	\end{figure}
	
	In the induction step, let $J \subset \Re^d$ be a $(d-2)$-flat. By the induction hypothesis there exist three pairwise disjoint convex sets $K_1',K_2',K_3' \subset J$ such that $|J \setminus (K_1' \cup K_2' \cup K_3')|=f(d-2)$. Let $\pi \subset \Re^d$ be a plane orthogonal to $J$, hence $\pi \cap J$ is a point. W.l.o.g., assume that $\pi \cap J=\{0\}$. Consider 3 points $x_1,x_2,x_3 \in \pi$ equally distributed around 0 (see Figure \ref{fig:fig18}). Consider the $(d-1)$-flat $H_1=\mbox{aff}(J,x_1)$, and let $H_1^+$ be the open half-flat of $H_1$ supported by $J$ that contains $x_1$. (Formally, $H_1^+=\{y + \lambda x_1: y \in J, \lambda>0\}$.) Define $H_2^+, H_3^+$ similarly.
	
	We are supposed to produce three pairwise disjoint convex sets ${K_1},{K_2},{K_3}$, such that $K_i' \subset {K_i}$ for $i=1,2,3$ and $\Re^d \setminus ({K_1}\cup {K_2} \cup {K_3})=J \setminus (K_1' \cup K_2' \cup K_3') \cup \{x_1,x_2,x_3\}$. By Corollary \ref{cor:extend}, $K_1'$ and $K_2'$ can be extended to two disjoint convex sets $K_1^3$ and $K_2^3$ such that $K_1^3 \cap J = K_1'$, $K_2^3 \cap J = K_2'$ and $K_1^3 \cup K_2^3= K_1' \cup K_2' \cup H_3^+ \setminus \{x_3\}$. Similarly, $K_1'$ and $K_3'$ can be extended to two disjoint convex sets $K_1^2$ and $K_3^2$ such that $K_1^2 \cap J = K_1'$, $K_3^2 \cap J = K_3'$ and $K_1^2 \cup K_3^2= K_1' \cup K_3' \cup H_2^+ \setminus \{x_2\}$. $K_2^1 $ and $K_3^1$ are obtained in the same way. Finally, we define ${K_1}= K_1^2 \cup K_1^3 \cup \mbox{int} (\mathrm{conv}(H_2^+ \cup H_3^+))$ and similarly, ${K_2}= K_2^1 \cup K_2^3 \cup \mbox{int} (\mathrm{conv}(H_1^+ \cup H_3^+))$ and ${K_3}= K_3^1 \cup K_3^2 \cup \mbox{int} (\mathrm{conv}(H_1^+ \cup H_2^+))$. The sets ${K_1}, {K_2}$ and ${K_3}$ encapsulate all the $f(d-2)$ points that were originally encapsulated by $K_1',K_2',K_3'$ in $J$, and additionally encapsulate $x_1,x_2,x_3$. Therefore, we obtain $f(d-2)+3=f(d)$ points that are encapsulated by 3 convex sets in $\Re^d$. This completes the proof of part (c) of Theorem \ref{thm:3encapsuled}, and hence the proof of Theorem \ref{thm:3encapsuled}.    
\end{proof}
\begin{figure}[ht]
	\begin{center}
		\scalebox{0.7}{
			\includegraphics[width=0.8\textwidth]{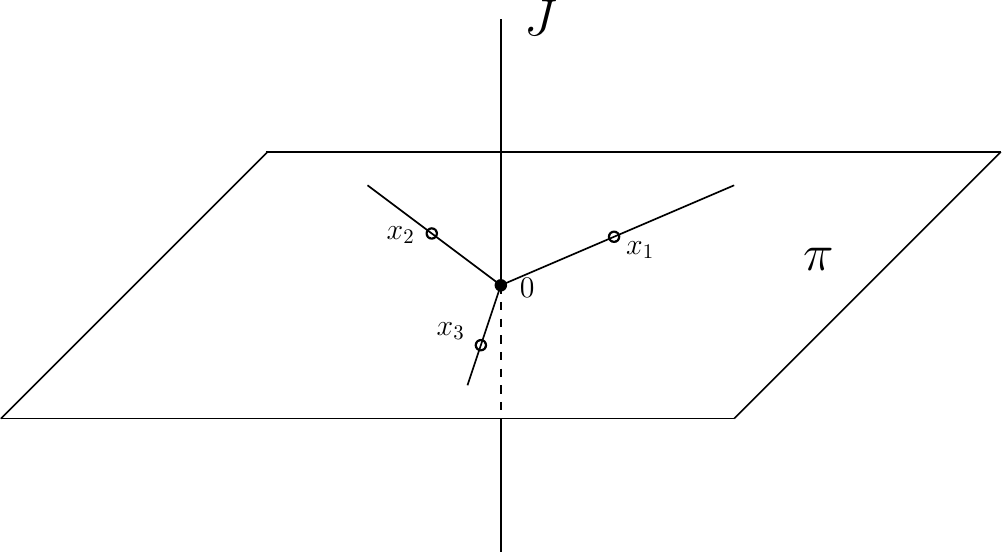}
		}
		\caption{An illustration for the induction step in the proof of Theorem \ref{thm:3encapsuled}(c).}
		\label{fig:fig18}
	\end{center}
\end{figure}

\subsection{The number of points encapsulated by $n$ convex sets in $\Re^d$}\label{subsec:n-encapsulated}
Recall that $f(d,n)$ is the maximal number of points that can be encapsulated by $n$ convex sets in $\Re^d$. Clearly\footnote{See also Theorem \ref{thm:two_sets-intro} below.}, $f(d,1)=0,f(d,2)=1$, and for all $n \geq 1$ we have $f(0,n)=0, f(1,n)=n-1$. After obtaining the value $f(d,3)= \lfloor \frac{3d}{2} \rfloor+1$, the following theorem yields a recursive upper bound on the function $f(d,n)$ in general.

\medskip

\noindent \textbf{Theorem \ref{thm:recursive} - restatement.}
	For every $n > 3,d \geq 2$
	\begin{equation}\label{eq:3}
	f(d,n) \leq   \sum_{i=2}^{n-1} \left(\binom{n}{i} f(d,i) \right) +f(d-2,n).
	\end{equation}
	Consequently,
	\begin{equation}\label{eq:3.1}
		f(d,n) \leq 2d^{n-2}n!.
	\end{equation}
\medskip 

\begin{proof}
	The proof of the recursive formula is by a double induction on $n$ and $d$, where the basic cases $f(1,n),f(d,2)$ and $f(d,3)$ have already been proved above. In the induction step, note that for each of the $\binom{n}{i}$ $i$-subsets of $K_1,\ldots,K_n$, the points of $\Re^d \setminus \bigcup_{i=1}^nK_i$ that touch exactly this subset, are encapsulated by this subset, and by the induction hypothesis their number is bounded by $f(d,i)$. Therefore it remains to prove that the number of points that touch all $n$ sets is at most $f(d-2,n)$.
	
	Indeed, Let $S'$ be the set of points that touch $K_1,\ldots,K_n$, and let $J=\mbox{aff}(S')$. By Observation \ref{obs:cone}, $\dim (J) <d$. If $\dim (J) \leq d-2$, we are done. 
	
	Otherwise, $\dim (J) = d-1$. In this case, as in the argument of Case 3 in the proof of Theorem \ref{thm:3encapsuled}, no point that is encapsulated by some subset of  $K_1,\ldots,K_n$ lies outside $J$.
	Indeed, otherwise there exists $T \subset[n]$ and a point $p \in \Re^d \setminus J$ that touches each $K_i$ for $i \in T$ and no $K_j$ for $j \notin T$. Then $\dim (\mathrm{conv}(S' \cup \{p\}))=d$, and since all the points in $S'$ touch all the $K_i$'s for $i \in T$, we have $\mathrm{conv}(S' \cup \{p\}) \subset \bigcap_{i \in T}\mbox{cl}(K_i)$, 
	and $\bigcap_{i \in T}K_i$ is $d$-dimensional. But then by Observation \ref{obs:cone}, $p$ is not encapsulated by $\bigcup_{i \in T}K_i$.
	
	Consequently, for every $T \subset [n]$, all the points that are encapsulated by $\bigcup_{i \in T}K_i$ are in $J$, and the bound $f(d-1,n)$, which is smaller than the right hand side of (\ref{eq:3}) by the induction hypothesis on $d$, follows.
	
	\medskip The upper bound~\eqref{eq:3.1} follows from the recursive formula~\eqref{eq:3} by induction on $n$, for every fixed $d \geq 2$. The induction bases are the bounds $f(d,2)=1$ and $f(d,3) \leq \lfloor \frac{3d}{2} \rfloor +1 \leq 2d$, proved above. Assume that~\eqref{eq:3.1} holds for all $f(d,n')$, $n'<n$. Hence, by~\eqref{eq:3} we have
	\[
	f(d,n) \leq \sum_{i=2}^{n-1} \left(\binom{n}{i} f(d,i) \right) +f(d-2,n) \leq    \sum_{i=2}^{n-1} \left(\binom{n}{i} \cdot 2d^{i-2}i! \right) + 2(d-2)^{n-2}n!.
	\]
	Define $h_{d,n}(i) = \binom{n}{i} \cdot 2d^{i-2}i!$. Note that for any $d \geq 2$, $n \geq 4$, $2 \leq i \leq n-2$, 
	\[
	\frac{h_{d,n}(i+1)}{h_{d,n}(i)}= \frac{\binom{n}{i+1} \cdot 2d^{i-1}(i+1)!}{\binom{n}{i} \cdot 2d^{i-2}i!} = \frac{n-i}{i+1} \cdot d(i+1) = (n-i)d \geq 2.
	\]
	Hence, $\sum_{i=2}^{n-1} h_{d,n}(i) \leq 2h_{d,n}(n-1)$. Thus, we have    	 
	\begin{align*}
	f(d,n) &\leq \sum_{i=2}^{n-1} \left(\binom{n}{i} \cdot 2d^{i-2}i! \right) + 2(d-2)^{n-2}n! \leq 2nd^{n-3}(n-1)!+ 2(d-2)^{n-2}n! \\
	&= 2n!d^{n-2} \left(\tfrac{2}{d}+(\tfrac{d-2}{d})^{n-2} \right) \leq 2n!d^{n-2}.
	\end{align*}
	This completes the proof of~\eqref{eq:3.1} by induction.
\end{proof}

We note that the upper bound~\eqref{eq:3.1} can be improved further with no much effort. However, it is not far from the best one can get from the recursion~\eqref{eq:3}. Indeed, it is easy to see that for $n$ constant, the dependence on $d$ in~\eqref{eq:3} is $\Omega(d^{n-2})$, and that for $d$ constant, the dependence on $n$ in~\eqref{eq:3} is super-exponential.

\section{Covering the complement of finite unions of convex sets by flats}\label{sec:covering} 

In this section we study covering of $S=\Re^d \setminus (\cup_{i=1}^n K_i)$ by flats -- i.e., affine subspaces of $\Re^d$. The main result we prove is Theorem~\ref{thm:local} which asserts that there exists a covering of $S$ by $g(n,d)$ flats such that any $p \in S$ is covered by a flat whose dimension is $dm(S,p)$.

\subsection{Preliminaries}

Recall that by Definition	\ref{def:dm},
 the flat-dimension of  $S$ is $dm(S)=\mbox{max}\{k: S \mbox{ includes a } k \mbox{-simplex}\}$, and 
the local dimension of $p \in S$  is $dm(S,p)=\min_{\epsilon>0}dm(S \cap B(p,\epsilon))$.

Define $dm(\emptyset)=-1$.
Note that $dm(S)=0$ iff $S$ is non-empty, but does not include a non-degenerate straight line segment, and $dm(S)=d$ iff $\mbox{int}(S) \neq \emptyset$.
 Note that for general sets $S \subset \Re^d$, $dm(S)$ does not necessarily coincides with the topological dimension of $S$, nor with the dimension of its affine hull, $\mbox{dim}( \mbox{aff}(S))$.

  As for the local dimension, $dm(S,p)=k$ if the intersection of $S$ with every neighborhood of $p$ includes a $k$-simplex, but the intersection of $S$ with some neighborhood of $p$ does not include a $(k+1)$-simplex. Refer to Figure \ref{fig:fig3} (and also to Figure \ref{fig:fig4} below) for an example of the local dimension.  Note that if $dm(S,p)=0$ then there exists some neighborhood $U$ of $p$ such that $S \cap U$ contains no segment. In this case $p$ is an isolated point of $S$, since by Theorem \ref{thm:local}, $S$ can be covered by finitely many points, hence $S \cap U$ is finite and therefore $p$ is isolated.

  	\begin{figure}[ht]
  	\begin{center}
  		\scalebox{0.5}{
  			\includegraphics[width=0.8\textwidth]{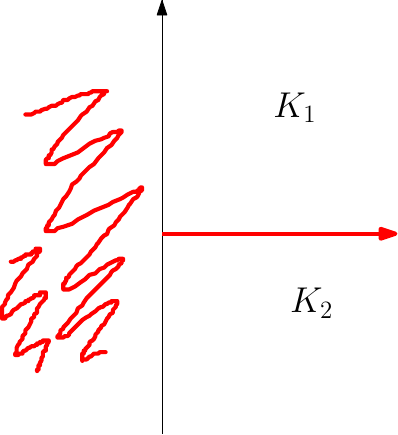}
  		}
  		\caption{In the figure, the red-colored area is $S=\{ (x,y):x<0 \}  \cup \{(x,0): x \geq 0\} = \Re^2 \setminus (K_1 \cup K_2) $, where $K_1=\{ (x,y):x \geq 0 , y>0\} $ and $K_2=\{ (x,y):x \geq 0 , y<0\} $. In this case, $dm(S,(0,0))=2$ and $dm(S,(1,0))=1$. The point $(1,0)$ is an ordinary point of $S$ (with the $x$-axis as the corresponding 1-flat), while $(0,0)$ is not. }
  		\label{fig:fig3}
  	\end{center}
  \end{figure}

 	\paragraph{Ordinary points and ordinary flats.}
 \begin{definition}\label{def:ordinary}
 	 The point $p$ is a $k$-ordinary point of $S$ ($p \in S \subset \Re^d$) if 
 	\begin{enumerate}
 		\item $dm(S,p) = k.$
 		\item For some $k$-flat $H \subset \Re^d$ and for some open neighborhood $U$ of $p$, $S \cap U= H \cap U$.
 	\end{enumerate}
  We call $H$ a $k$-ordinary flat of $S$. Moreover, $p$ is called an ordinary point of $S$ if it is $k$-ordinary for some $0 \leq k \leq d$. Similarly for flats.
\end{definition}
The set of $k$-ordinary points of $S$ is a relatively open subset of $S$. Indeed, if $p$ is a $k$-ordinary point of $S$ and $U,H$ are as in Definition \ref{def:ordinary}, then every point $x \in S \cap U$ is $k$-ordinary as well with the same $k$-flat $H$. 
Note that if Definition \ref{def:ordinary} holds for $U$ then it clearly holds for any smaller neighborhood $p \in U' \subset U$. Moreover, the flat $H$ is determined by the point $p$, $H=\mbox{aff} (U \cap S)$, for an appropriate neighborhood $U$ as in Definition \ref{def:ordinary}.

\begin{remark}
	If $S'$ is relatively open subset of $S$, (i.e., $S'=S \cap U$ where $U$ is an open subset of $\Re^d$) then for every point $p \in S'$, $dm(S,p)=dm(S',p)$. Moreover, $p$ is a $k$-ordinary point of $S$ iff it is a $k$-ordinary point of $S'$ (with the same $k$-ordinary flat.)
\end{remark}

\paragraph{Strong covers.} 
\begin{definition}
	Let $S \subset \Re^d$. A collection $\mathcal{C}$ of flats in $\Re^d$ is a strong cover of $S$, if:
	\begin{enumerate}
		\item $|\mathcal{C}|< \infty$
		\item Each point $p \in S$ belongs to some $k$-flat $H \in \mathcal{C}$, where $k=dm(S,p)$.
		\item $\mathcal{C}$ is minimal, i.e., no proper subcollection of $\mathcal{C}$ satisfies (2).
	\end{enumerate}
\end{definition}

\subsection{Proof of Theorem~\ref{thm:local}}

 Our main result in this section is the following restatement of Theorem~\ref{thm:local}.
 
\medskip \noindent \textbf{Theorem~\ref{thm:local} -- restatement.}
 	Let $K_1,\ldots,K_n$ be convex subsets of $\Re^d$ and let $S=\Re^d \setminus (\cup_{i=1}^n K_i)$.  Then $S$ admits a unique strong cover $\mathcal{C}$. Actually, $\mathcal{C}$ is the collection of all ordinary flats with respect to $S$.

\medskip

Note that $p$ is encapsulated by $K_1, \ldots, K_n$ if and only if $\{p\}$ is an 0-ordinary flat w.r.t. $ \Re^ d \setminus( K_1 \cup \ldots \cup K_n)$. 
Hence, both Theorems \ref{thm:3encapsuled} and \ref{thm:recursive} above, supply quantitative bounds for special cases of Theorem \ref{thm:local}.

It is easy to see that any strong cover $\mathcal{C}$ of $S$ must contain all ordinary flats of $S$. Indeeed, given a $k$-ordinary flat $H$, there exists some $k$-ordinary point $p \in S$ and some open neighborhood $U$ of $p$, such that $S \cap U= H \cap U$. If  $H \notin \mathcal{C}$ then
consider all the $k$-flats  $H' \in \mathcal{C}$. Since for any such $H'$, $H \neq H'$, it follows that $\mbox{dim} (H \cap H') <k$, therefore the neighborhood $U$ is not covered by these (finitely many) $k$-flats $H'$. But as mentioned above, each point in $S \cap U$ is a $k$-ordinary point of $S$ -- a contradiction.

 %$p$ is covered by some $H' \in \mathcal{C}$ with $\mbox{dim} H'=k$. But since $H \neq H'$ it follows that $p \in H \cap H'$ and $\mbox{dim} (H \cap H') <k$, a contradiction.

Therefore, in order to prove Theorem \ref{thm:local} we show that there are only finitely many ordinary flats with respect to $S$, and that condition (2) of Definition \ref{def:ordinary} is satisfied not only for the ordinary points of $S$ (which is trivial), but also for the non-ordinary points of $S$.

\medskip

\begin{proof}[Proof of Theorem \ref{thm:local}]
	The proof is by a primary induction on $d$ and a secondary induction on $n$. %Namely, given $(n,d)$ we assume correctness for $(n,d')$ for all $d'<d$ and correctness for $(n',d)$ for all $n'<n$. 
	For the induction basis, let $p \in S \subset \Re^d$ with $dm(S,p)=k$. If $k=d$ then no inductive argument is needed since $p \in \mbox{cl}(\mbox{int}(S))$, and then the corresponding $d$-flat is $\Re^d$. The case $d=0$ is trivial. If $d=1$ and $k=0$ then $p$ is an isolated point of $S$ and the number of such points $p$ is at most $n-1$. The corresponding 0-flat in this case is clearly $\{p\}$.
	 If $d=1$ and $k=1$ then $\Re^1$ is the corresponding 1-flat. If $n=1$ then for any $p \in S$, $dm(S,p)=d$ and $\Re^d$ is the $d$-ordinary flat.
	 
	% If $n=2$ and $p$ does not touch $K_1$ (resp. $K_2$) then $dm(S,p)=d$ and $\Re^d$ is the ordinary $d$-flat. If $p$ touches both $K_1$ and $K_2$ where $\mbox{int}(K_1) \cap \mbox{int}(K_2) \neq \emptyset$, or if $\mbox{dim}(K_1)<d$ (resp.  $\mbox{dim}(K_2)<d$), then $dm(S,p)=d$ and by Observation \ref{obs:cone} the corresponding $d$-flat is $\Re^d$.
	 
	% The remaining case is where $p$ touches both $K_1$ and $K_2$, $\mbox{int}(K_1) \cap \mbox{int}(K_2) = \emptyset$ and $\mbox{dim}(K_1)=\mbox{dim}(K_2)=d$. In this case $K_1$ and $K_2$ can be separated by a hyperplane $H$ through $p$. We can assume that there exists a neighborhood $U$ of $p$ such that $U=(K_1 \cap U) \cup (K_2 \cap U) \cup (H \cap U)$ (since otherwise in any neighborhood of $p$ there is a point outside $H$, and then we come back to the case $n=1$ and the local dimension is $d$). Then $S \cap H = S \cap U$ and we are done by the induction hypothesis for $n=2$ and $d-1$.

	% If $d=1$ and $k=1$ for some $p \in S$, then $\Re^1$ is an ordinary 1-flat, and if $d=1$ and $k=0$ for every $p \in S$, then $S$ is a finite union of points, and each of them is an ordinary 0-flat w.r.t. $S$. If $n=1$ then for any $p \in S$, $dm(S,p)=d$ and $\Re^d$ is the ordinary $d$-flat. If $n=2$ and $\mbox{int}(K_1) \cap \mbox{int}(K_2) \neq \emptyset$ then by Observation \ref{obs:cone} the corresponding $d$-flat is $\Re^d$. If $n=2$ and $\mbox{int}(K_1) \cap \mbox{int}(K_2) = \emptyset$ then $K_1$ and $K_2$ can be separated by a hyperplane $H$, and the corresponding ordinary flats are  $\Re^d, H$ . 
	
	For larger values of $n$ and $d$ we shall prove by induction the following:
	
	\begin{enumerate}
		\item \label{cond1} For any $0 \leq k \leq d$, there are finitely many $k$-ordinary flats w.r.t.~$S$. 
		\item \label{cond2} For any $p \in S$ with $dm (S,p)=k$ and for any $\epsilon >0$, there is a $k$-ordinary point $q \in S$ such that $||p-q||<\epsilon$.
	\end{enumerate} 

\begin{remark}
	The second statement implies the existence of a $k$-ordinary flat $J$ such that $p \in J$. Indeed, by the statement $p$ is in the closure of the $k$-ordinary points of $S$, each of which is contained in a $k$-ordinary flat. But these flats form a closed set, hence $p$ is contained in at least one of these $k$-ordinary flats.
\end{remark}

The sets to which we apply the induction hypothesis are as follows: 
\begin{itemize}
	\item $S^i = \Re^d \setminus \bigcup_{j \neq i}K_j$, ($i=1,2,\ldots,n$), to which we apply the induction hypothesis with the same $d$ and $n-1$ convex sets.
	\item For a flat $H$ (to be specified later) with $\mbox{dim}(H)<d$, consider $S^H=H \setminus (\bigcup_{j=1}^n(K_j \cap H))$, that is, the restriction\footnote{Note that along the proof, when considering $S^H$, we treat $S^H$ as a subset of $H$ (and not as a subset of $\Re^d$).} of the original system to $H$. Here we apply the induction hypothesis with the dimension $\mbox{dim}(H)$.
\end{itemize}

We will prove that any $k$-ordinary flat w.r.t.~$S$ is either a $k$-ordinary flat w.r.t.~some $S^i$ or a $k$-ordinary flat for some $S^H$ with $\mbox{dim}(H)<d$. Hence we will be able to apply the induction hypothesis either to $S^i$ or to $S^H$. 

%Let $J$ be an ordinary $k$-flat that corresponds to some $k$-ordinary point $p \in S$. Namely, there exists a neighborhood $U$ of $p$ such that $S \cap U = J \cap U$. 
Let $p \in S$ with $dm(S,p)=k$. There are two options:

\begin{itemize}
	\item \textbf{Case 1:} For some $1 \leq i \leq n$, $p$ does not touch $K_i$.
	\item \textbf{Case 2:} $p$ touches all sets $K_i$, $1 \leq i \leq n$.
\end{itemize}
We handle each case separately.

\medskip

\noindent \textbf{Case 1:} Since $p$ does not touch $K_i$, it follows that $p \in \mbox{int}(\Re^d \setminus K_i)$, and thus, $p$ has a positive distance from $K_i$. Let $U$ be a neighborhood of $p$ in $\Re^d$ that satisfies  $U \subset \mbox{int}(\Re^d \setminus K_i)$. Then $S \cap U = S^i \cap U$ and $dm(S^i,p)=dm(S,p)=k$. Moreover, for every corresponding $k$-flat $J$ with $J \cap U = S \cap U$ we have  $J \cap U = S^i \cap U$. Hence, the set of $k$-ordinary flats w.r.t.~$S$ that correspond to a point $p$ as in case 1, is included in the union of the sets of $k$-ordinary flats w.r.t. $S^i$, $1 \leq i \leq n$, and we are done by the induction hypothesis on the  $S^i$'s.

%. By the induction hypothesis on $S^i$, there exists a $k$-flat $J$ such that $J \cap U =S^i \cap U = S \cap U$, and again by the induction hypothesis, both (\ref{cond1}) and (\ref{cond2}) are satisfied, since the set of $k$-ordinary flats w.r.t. $S$ is contained in the union of the sets of $k$-ordinary flats w.r.t. $S^i$, ($1 \leq i \leq n$), each of which is finite. 
% and by shrinking $U$ we can assume that $U \cap K_i = \emptyset$. Therefore, $U \setminus (J \cap U)=U \setminus (S \cap U)$ is covered by $\bigcup_{1 \leq j \leq n, j \neq i}K_j$. Note that inside $U$, $S$ and $S^i$ coincide, hence $dm(S^i,p)=dm(S,p)=k$. Therefore, $J$ must participate as an ordinary $k$-flat (w.r.t. $p$) in any cover of $S^i$ by $k$-flats, and by the induction hypothesus on $n$, there are only finitely many such $J$'s.

\medskip

\noindent \textbf{Case 2:} Let $G=\{p \in S : p \mbox{ touches all }K_i\mbox{'s}\}$ and let $H=\mbox{aff}(G)$ be the flat that is spanned by $G$. If $H=\Re^d$ then $\mbox{conv}(G)$ is $d$-dimensional and $\mbox{int(conv}(G)) \subset \bigcap_{1 \leq i \leq n}K_i$. Then $dm(p,S)=d$ holds for all $p \in G$, and the corresponding $d$-flat is $\Re^d$. From now on we assume $\mbox{dim}(H)<d$. 

 Clearly $dm(S^H,p) \leq k$, and every neighborhood of $p$ in $S$ includes a $k$-simplex. We distinguish between two cases:
 
	\noindent \textbf{Case 2a: Any neighborhood $U$ of $p$ includes a $k$-simplex that is not included in $H$.} In this case, in each sufficiently small neighborhood $U=B(p,\frac{1}{m})$ of $p$ we can find a point $q_m$ with $dm(S,q_m)=k$ and $q_m \notin H$. By the definition of $H$, this means that for some $1 \leq i \leq n$, $q_m$ does not touch $K_i$. By passing to a subsequence $\{q_{m_r}\}_{r=1}^{\infty}$ we can assume that for a fixed $1 \leq i \leq n$, each $q_{m_r}$ does not touch $K_i$. By the induction hypothesis on $S^i$, each $q_{m_r}$ is contained in a $k$-ordinary flat w.r.t.~$S^i$, where the number of such flats is finite. By passing again to a subsequence, we can restrict ourselves to only one such $k$-ordinary flat, $J$. Therefore, we obtain in this flat a sequence of $k$-ordinary points that tend to $p$ and by the closedness of $J$, $p \in J$.
	
	\medskip   

\noindent \textbf{Case 2b: There exists a neighborhood $U$ of $p$ such that every $k$-simplex in $S \cap U$ is contained in $H$.} 
In this case, $dm(S^H,p)=k$ and we can apply the induction hypothesis to $S^H$, since we have already assumed $\mbox{dim}(H)<d$.
By this induction hypothesis there exists a $k$-ordinary flat $J'$ (w.r.t $S^H$) that contains $p$, and a sequence of $k$-ordinary points in $J' \cap S^H$ that tends to $p$.

The only thing that is left to prove is that $J'$ is ordinary not only w.r.t.~$S^H \subset H$ but also w.r.t.~$S \subset \Re^d$. 
Indeed, let $U' \subset U$ be a neighborhood of $p$ such that $U' \cap J' = U' \cap S^H$. If some point $q \in U' $ is the limit of a sequence of points from $S$ outside $H$, each point in the sequence does not touch some $K_i$, and by passing to a subsequence, we can assume that all its elements do not touch $K_i$ for a fixed $i$. Then, by the induction hypothesis on $S^i$, the whole sequence is contained in finitely many $(\leq k)$-flats, none of them is $J'$. The union of all these $(\leq k)$-flats, intersects $J'$ in a $(< k)$-dimensional set, hence we can find $U'' \subset U'$ s.t.  $U'' \cap J' = U'' \cap S$, namely, $J'$ is $k$-ordinary also w.r.t.~$S$.    
\end{proof}

\section{Bounding the number of ordinary hyperplanes}
\label{sec:hyperplanes}
%The proof of Theorem \ref{thm:K1K2} implies the existence of a strong cover of the complement $S$ of $n$ convex sets $K_1, \ldots, K_n \subset \mathbb{R}^d$, with  finitely many flats, but does not supply efficient bounds on the number of these flats. In Theorems \ref{thm:3encapsuled} and \ref{thm:recursive} above we proved upper bounds in the special case of $k=0$. While obtaining a tight bound in general seems to be a difficult problem, in this section we consider a setting in which a tight bound on the number of $k$-flats from Theorem \ref{thm:local} can be proved -- the setting of $(d-1)$-flats. 

In this section we present the proof of Theorem~\ref{thm:num_flats-intro} which essentially determines the maximal possible number of hyperplanes in the strong cover $\mathcal{C}$ of $S=\Re^d \setminus (\cup_{i=1}^n K_i)$, whose existence was proved in Theorem~\ref{thm:local}.  
In Theorem \ref{thm:num_flats} we prove the assertion of the theorem for $d \geq 3$, and in Theorem \ref{thm:num_flats_d=2} we prove it for $d=2$. We formulate these theorems using the notion of \emph{ordinary hyperplanes}, whose number is equal, by the proof of Theorem~\ref{thm:local}, to the notion $\nu_{d-1}$ used in the formulation of Theorem~\ref{thm:num_flats-intro}.
  
\begin{theorem}\label{thm:num_flats} 
	Let $d \geq 3$ and let $S=\Re^d \setminus (\cup_{i=1}^n K_i)$, where $\{K_i\}$ are convex sets. Then the number of ordinary hyperplanes (i.e., $(d-1)$-flats) w.r.t.~$S$ is at most $\binom{n}{2}$. On the other hand, the value $\binom{n}{2}$ is attained for some $K_1,\ldots,K_n$. 
\end{theorem}  

\medskip

\begin{proof}
	First, we prove that at most $\binom{n}{2}$ hyperplanes are needed for any $d \geq 2$. To this end, we prove that every ordinary hyperplane separates two $K_i$'s, and that it is the only hyperplane that separates them, hence the upper bound $\binom{n}{2}$ follows.
	
	Indeed, let $\pi$ be an ordinary hyperplane that corresponds to some $p \in S$. There exists a ball $B=B(p,\epsilon)$ with $B \cap \pi = B \cap S$. Let $B^+$ (resp., $B^-$) be the intersection of $B$ with the open half-space above (resp., below) $\pi$. W.l.o.g., $B^+$ is covered by $K_1 \cup \ldots \cup K_t$ and $B^-$ is covered by $K_{t+1} \cup \ldots \cup K_n$. Therefore, there exist $1 \leq i \leq t$ and $t+1 \leq j \leq n$ such that $\mbox{dim} (\mbox{cl}{K_i} \cap (B \cap \pi)) = \mbox{dim} (\mbox{cl}{K_j} \cap (B \cap \pi) = d-1)$. Hence, $\pi$ is the only hyperplane that sepaprates $K_i$ and $K_j$. Thus, the upper bound $\binom{n}{2}$ follows.
	
	On the other hand, in dimension $d+1 \geq 4$ there exists a neighborly polytope with $n$ vertices where each pair of vertices is connected by an edge. The dual polytope, $P$, has $n$ facets, where every two facets intersect in a $(d-1)$-face. Assume that $x$ is the (unique) highest vertex of $P$, and apply a central projection from $x' \in \mbox{int}(P)$ which is very close to $x$ (above all other vertices of $P$), to some $d$-dimensional horizontal hyperplane $\pi$ that lies strictly below $P$. The image of each facet that does not contain $x$ is a $d$-dimensional polytope in $\pi$, and the facets that contain $x$ are projected to unbounded $d$-dimensional polyhedral sets in $\pi$. These images of the facets of $P$ in $\pi$ are $n$ convex sets $\hat{K_1}, \ldots,\hat{K_n}$ with $\hat{K_1}\cup \ldots \cup \hat{K_n}=\pi$, and for every $1 \leq i<j \leq n$, $\mbox{dim}(\hat{K_i} \cap \hat{K_j})=d-1$.
	Now, let $K_1 = \mbox{int}(\hat{K_1}), \ldots, K_n = \mbox{int}(\hat{K_n})$. The strong cover of $S= \pi \setminus \bigcup_{i=1}^n K_i$ consists of the $\binom{n}{2}$ $(d-1)$-flats $\mbox{aff}(\hat{K_i} \cap \hat{K_j})$, $1 \leq i < j \leq n$. Hence, the value $\binom{n}{2}$ is obtained.
\end{proof}

\medskip

Actually, Theorem \ref{thm:num_flats} above implies the upper bound $\binom{n}{2}$ already for $d \geq 2$, and a construction with $\binom{n}{2}$ hyperplanes only for 
$d \geq 3$, since neighborly polytopes (other than simplices) exist only for $d \geq 4$. Similarly to the second part of the proof of Theorem \ref{thm:num_flats}, one can obtain
\footnote{Recall that $k$-neighborly $d$-polytopes with $n$ vertices ($d+1<n<\infty)$ exist for $d \geq 2k$. The construction described above leads to a tesselation of $\Re^d$ ( $d \geq 2k>1$) with $n$ polyhedral cells, where each $k$ cells meet in a $(d-1+k)$-face. Removing the  $(d-1+k)$-skeleton of this tesselation, we obtain $n$ convex sets $K_1, \ldots,K_n$ with $S=\Re^d \setminus (K_1 \cup \ldots \cup K_n)$, $\mbox{dim}(S)=d+1-k$, and the number of  $(d+1-k)$-flats in it is $\binom{n}{k}$. Note that by a tiny perturbation we can ensure that two different faces of the same dimension are projected into different faces.}
 a construction with $\binom{n}{3}$  $(d-2)$-flats where $d \geq 5$, and in general, $\binom{n}{k}$ $(d+1-k)$-flats where $d \geq 2k-1$, but in these cases no matching upper bound is known.

\medskip

In the case $d=2$, we prove an upper bound of $(\frac{3}{4}+o(1)) \binom{n}{2}$, and provide a construction for which this value is obtained. 

\medskip

\begin{theorem}\label{thm:num_flats_d=2}
 	Let $S=\Re^2 \setminus (\cup_{i=1}^n K_i)$, where $\{K_i\}$ are convex sets. Then the maximum possible number of ordinary lines w.r.t.~$S$ is the Tur{\'{a}}n number $t(n,4)=(\tfrac{3}{4}+o(1))\binom{n}{2}$. On the other hand, this value is attained for some $K_1,\ldots,K_n$.
\end{theorem}

\begin{proof} 
	For the upper bound we consider the geometric graph $G$, whose vertices are $n$ arbitrary points $x_i  \in \mbox{int}(K_i)$ ($1 \leq i \leq n, \forall i \neq j: x_i \neq x_j$), and whose edges are defined as follows:   
	For each ordinary line $\ell$ that corresponds to a 1-ordinary point $p_{\ell}$, there exist (as in the proof of Theorem \ref{thm:num_flats}) $K_i,K_j$ with $\dim K_i = \dim K_j=2$, such that $\ell$ is the only line that weakly separates $K_i$ and $K_j$, $p \in \mbox{cl}{(K_i)} \cap \mbox{cl}({K_j}) $. (Clearly, $\mbox{int}(K_i) \cap \mbox{int}(K_j) = \emptyset$.) We connect $x_i$ to $x_j$ in $G$ by a two-edges polygonal path  $[x_i,p_{\ell},x_j]$.
	
	By the disjointness of $\mbox{int}(K_i)$ and $\mbox{int}(K_j)$, the graph $G$ (though not necessarily planar) does not include $K_5$, since otherwise there exist 5 pairwise disjoint open convex sets whose closures are pairwise intersecting, which is impossible since the restriction of $G$ to the corresponding vertices is plannar by the disjointness of the 5 $\mbox{int}(K_i)$'s. Hence, by Tur\'{a}n's Theorem, $|E(G)|\leq t(n,4) = (\frac{3}{4}+o(1)) \binom{n}{2}$, and the upper bound on the number of ordinary lines follows.
	
	On the other hand, we present a construction 
	that realizes every complete 4-partite graph $K_{n_1,n_2,n_3,n_4}$ with $n_1+n_2+n_3+n_4=n$, in particular Tur\'{a}n's graph.
	 The construction is a bit complicated, hence we present it in five steps accompanied by illustrations.
	
	\medskip
	
	\noindent \textbf{Step 1:}
	First consider four closed convex sets $K_1,K_2,K_3,K_4$ as in Figure \ref{fig:fig7}. Let $e_{ij}=K_i \cap K_j$ for $1 \leq i < j \leq 4$.  Each segment $e_{ij}$ on the line that separates $K_i$ from $K_j$. 
	
	 \begin{figure}[ht]
		\begin{center}
			\scalebox{0.6}{
				\includegraphics[width=0.8\textwidth]{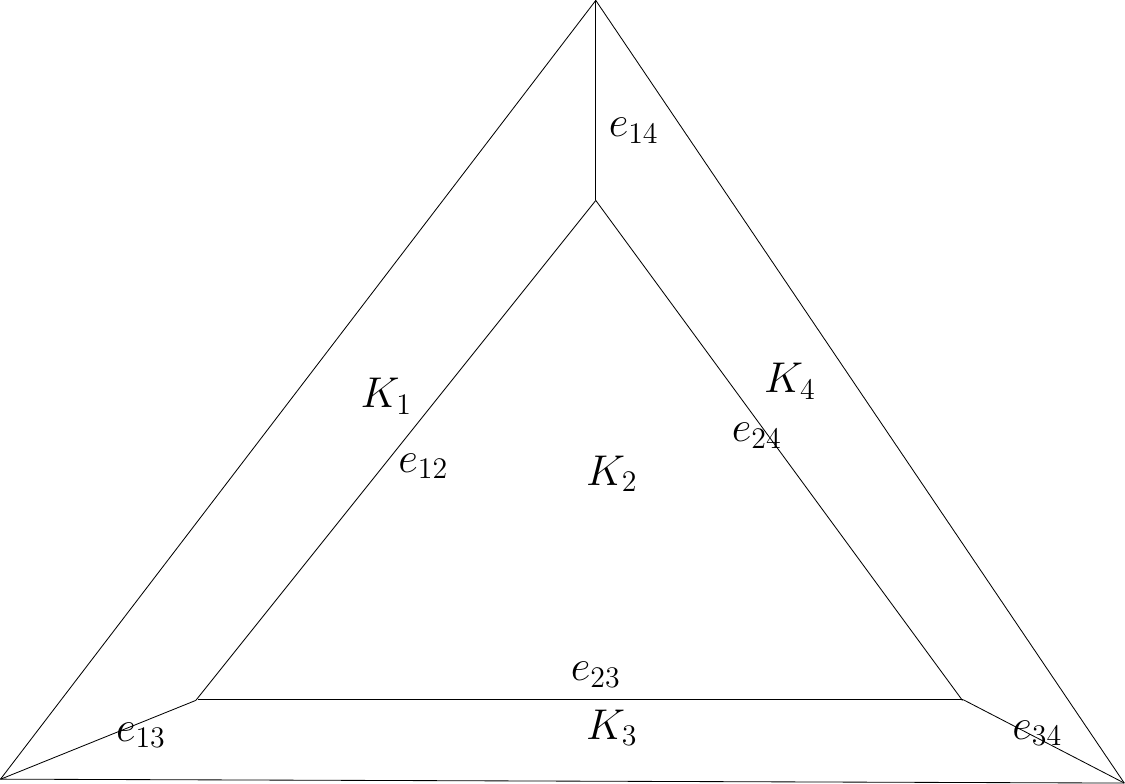}
			}
			\caption{An illustration for step 1 in the proof of Theorem \ref{thm:num_flats_d=2}.}
			\label{fig:fig7}
		\end{center}
	\end{figure}

\medskip
Our goal is to construct for each $1 \leq i \leq 4$, $n_i$ open convex sets $K_i^1,\ldots,K_i^{n_i}$, each of which is a subset of a tiny perturbation of $K_i$. This construction has the property that for every $1 \leq i <j \leq 4$, and for every $1 \leq \ell_i \leq n_i,1 \leq \ell_j \leq n_j$, there exists an ordinary line separating $K_i^{\ell_i}$ from $K_j^{\ell_j}$. By taking the $n_i$'s ($1 \leq i \leq 4$) as equal as possible and satisfying $n_1+n_2+n_3+n_4=n$, we obtain in this way $n$ convex sets with $\Sigma_{i<j} n_i n_j=t(n,4)$ ordinary lines as asserted.

\medskip

\noindent \textbf{Step 2:}
For every $1 \leq i<j \leq 4$, draw a circular arc $\gamma_{ij}$ through the endpoints of $e_{ij}$ inside $K_i$ (see Figure \ref{fig:fig8}).
	 \begin{figure}[ht]
	\begin{center}
		\scalebox{0.65}{
			\includegraphics[width=0.8\textwidth]{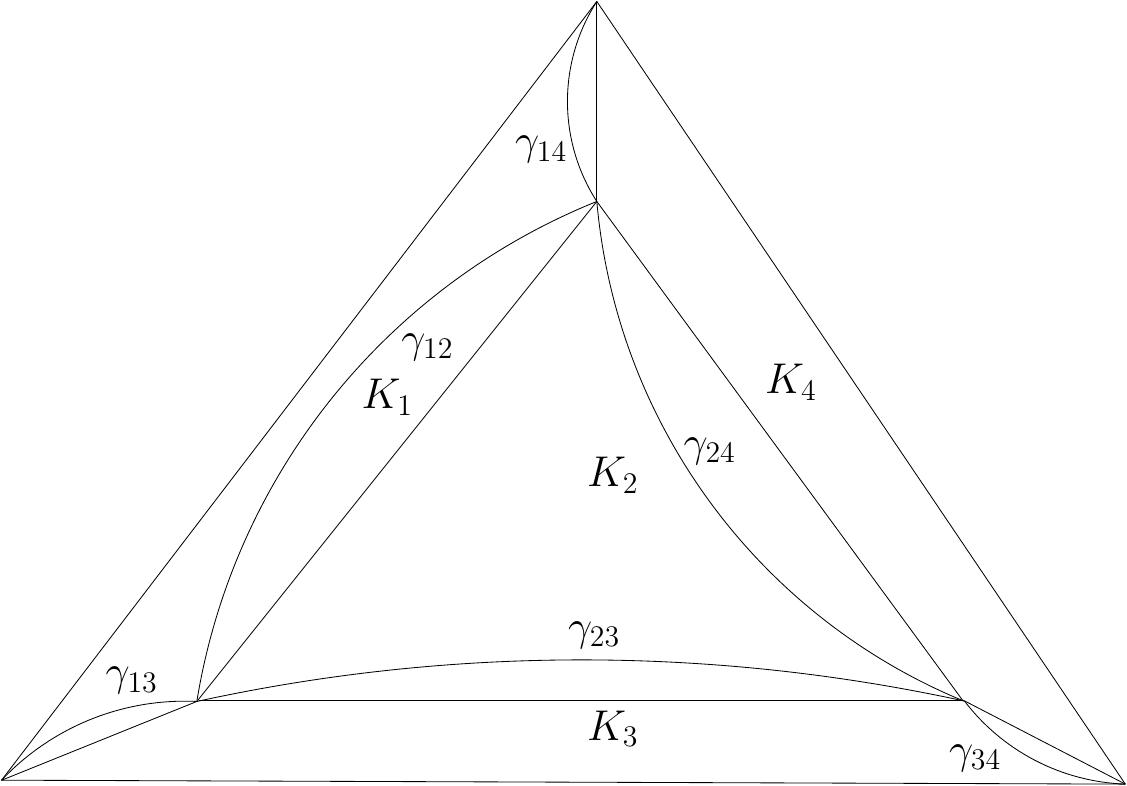}
		}
		\caption{An illustration for step 2 in the proof of Theorem \ref{thm:num_flats_d=2}.}
		\label{fig:fig8}
	\end{center}
\end{figure}
The radius of the circle associated with this arc, should be very large (in a sense to be clarified later). Note that in Figures \ref{fig:fig8}-\ref{fig:fig11} this radius is not large enough, just to make the illustration easier to follow.

\medskip

\noindent \textbf{Step 3:}
For every $1 \leq i < j \leq 4$, split $\gamma_{ij}$ by $n_i +1$ points into $n_i+2$ parts. Let the edges of the polygonal path through these points (and the endpoints of $e_{ij}$) be $a_0^{ij}, a_1^{ij}, \ldots, a_{n_i+1}^{ij}$. See Figure \ref{fig:fig9}, in which the polygonal paths corresponding to $\gamma_{12}, \gamma_{23}$ and $\gamma_{24}$ are illustrated. In this figure, $n_1=2$ and $n_2=3$. 
	 \begin{figure}[ht]
	\begin{center}
		\scalebox{1.0}{
			\includegraphics[width=0.8\textwidth]{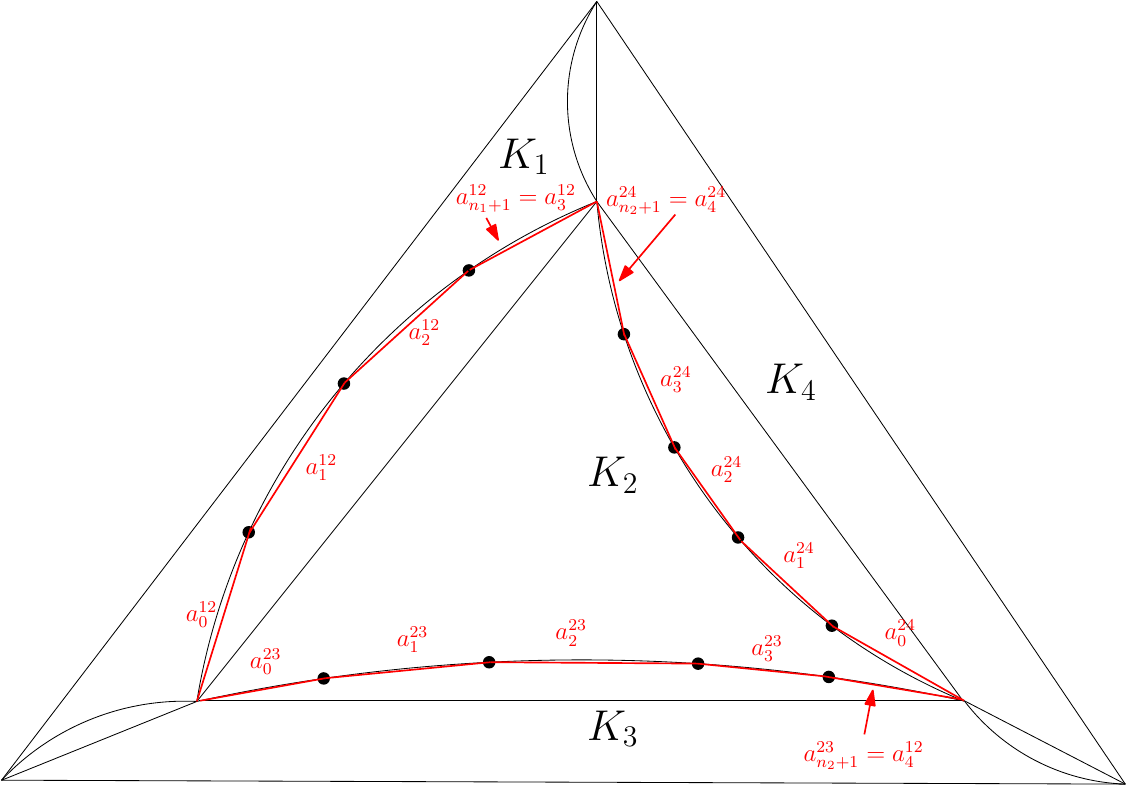}
		}
		\caption{An illustration for step 3 in the proof of Theorem \ref{thm:num_flats_d=2}. Here $n_1=2,n_2=3$.}
		\label{fig:fig9}
	\end{center}
\end{figure}

\medskip

\noindent \textbf{Step 4:}
For every $1 \leq i < j \leq 4$, and $1 \leq t  \leq n_i$, draw an arc $\delta_t^{ij}$ through the endpoints of $a_t^{ij}$ to the side of $K_j$. Again, the radius of the circle associated with $\delta_t^{ij}$ should be very large (as will be determined later), much larger than in Figure \ref{fig:fig10} (in which still $n_1=2$ and $n_2=3$).
	 \begin{figure}[ht]
	\begin{center}
		\scalebox{1.0}{
			\includegraphics[width=0.8\textwidth]{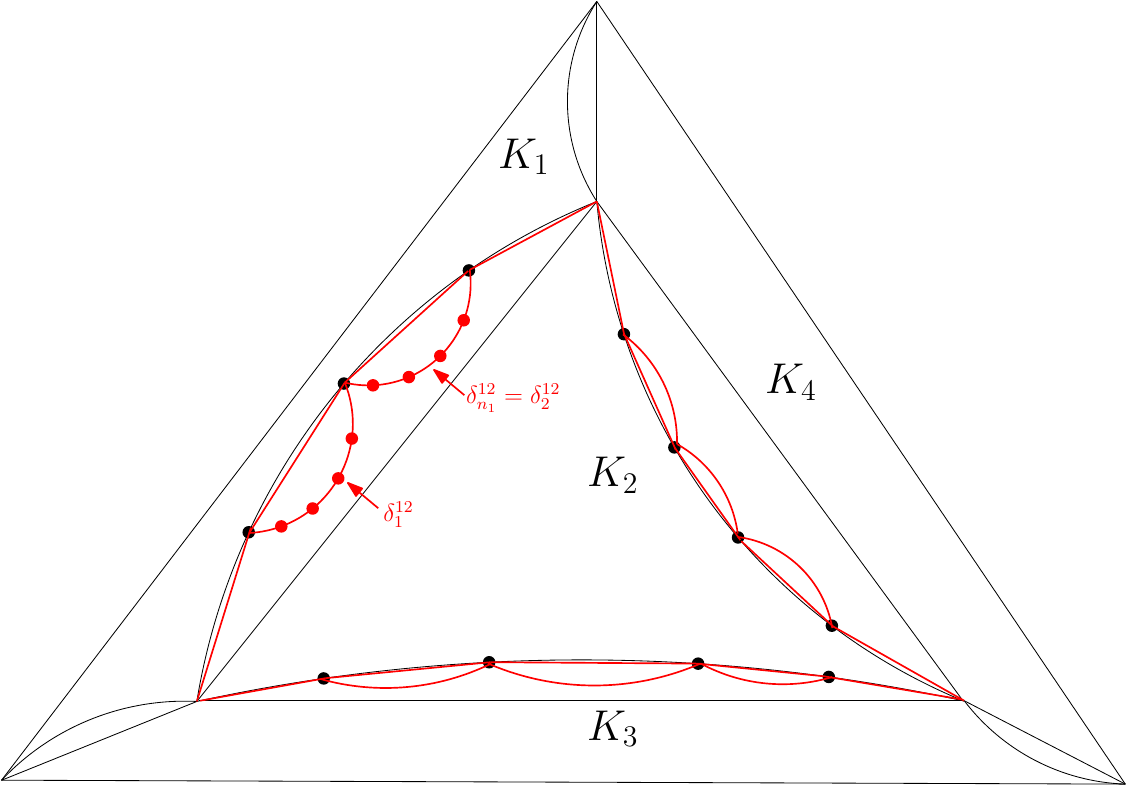}
		}
		\caption{An illustration for step 4 in the proof of Theorem \ref{thm:num_flats_d=2}. Here $n_1=n_3=n_4=2,n_2=3$.}
		\label{fig:fig10}
	\end{center}
\end{figure}
Partition $\delta_t^{ij}$ into $n_j+2$ sub-arcs by adding $n_j+1$ points on it, and let the edges of the polygonal path through these points (and the endpoints of $a_t^{ij}$) be $b_{t,0}^{ij}\ldots,b_{t,n_j+1}^{ij}$. In Figure \ref{fig:fig11}, these polygonal paths are illustrated, where $n_1=n_3=n_4=2$, and $n_2=3$. 
	 \begin{figure}[ht]
	\begin{center}
		\scalebox{1.2}{
			\includegraphics[width=0.8\textwidth]{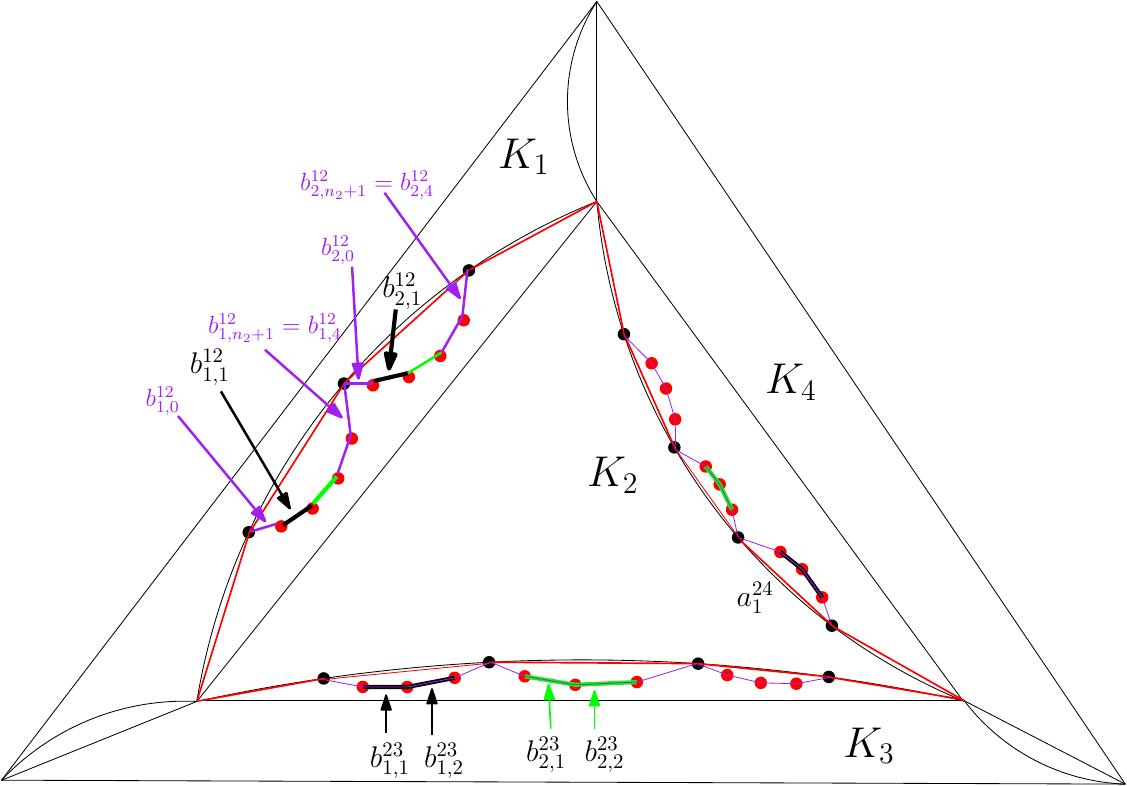}
		}
		\caption{An illustration for steps 4-5 in the proof of Theorem \ref{thm:num_flats_d=2}.  Here $n_1=n_3=n_4=2,n_2=3$.}
		\label{fig:fig11}
	\end{center}
\end{figure}

\medskip

\noindent \textbf{Step 5:}
Let  $1 \leq i  \leq 4$. In this step we construct $n_i$ open convex sets $K_i^1, \ldots,K_i^{n_i}$, each of which is a subset of a tiny perturbation of $K_i$. Each of these $n_i$ convex sets will be the interior of the convex hull of several segments, and the radii of the circles mentioned in steps 2 and 4 should be so large, such that all the segments that are involved in the construction of each $K_i^{\ell}$ ($1 \leq \ell \leq n_i$) will be on the boundary of $\mbox{cl} (K_i^{\ell})$. Note that each $K_i^{\ell}$ is an open convex polygon\footnote{This is the reason why in the construction we omit the first and the last segments of each arc.}.   

The first such set, $K_i^1$, is the interior of the convex hull of the following segments: 

$$ \begin{cases}
	b_{1,1}^{ij},	b_{2,1}^{ij}, \ldots 	b_{n_j,1}^{ij} & :1 \leq j<i \\
	b_{1,1}^{ij},	b_{1,2}^{ij}, \ldots 	b_{1,n_j}^{ij} & :i<j \leq 4
\end{cases}$$

In Figure \ref{fig:fig11} (that still demonstrates the case where $n_1=2,n_2=3$), the segments that form $K_2^1$ are colored black. Roughly speaking, from the arcs inside $K_i$ we take consecutive `$b$' edges whose endpoints lie on the first `$\delta$'-arc corresponding to each edge of $K_i$, and from the arcs out of $K_i$ that correspond to an edge of $K_i$, we take the first `$b$'-edge from each `$\delta$'-arc.
Recall that the radii of all arcs involved are so large, that (unlike the drawing in Figures \ref{fig:fig8} -\ref{fig:fig11}) all these `$b$'-edges that form each $K_i^{\ell}$ are in convex position.

The second such set, $K_i^2$, is the interior of the convex hull of the following segments, colored by green in Figure \ref{fig:fig11}:

 $$ \begin{cases}
 	b_{1,2}^{ij},	b_{2,2}^{ij}, \ldots 	b_{n_j,2}^{ij} & :1 \leq j<i \\
 		b_{2,1}^{ij},	b_{2,2}^{ij}, \ldots 	b_{2,n_j}^{ij} & :i<j \leq 4
 \end{cases}$$

In general, for each $1 \leq \ell \leq n_i $, the set $K_i^{\ell}$ is the interior of the convex hull of the following segments: 

 $$ \begin{cases}
	b_{1,\ell}^{ij},	b_{2,\ell}^{ij}, \ldots 	b_{n_j,\ell}^{ij} & :1 \leq j<i \\
		b_{\ell,1}^{ij},	b_{\ell,2}^{ij}, \ldots 	b_{\ell,n_j}^{ij} & :i<j \leq 4
\end{cases}$$

In this way, as was mentioned just after Step 1, we construct $n$ open convex sets with  
$\Sigma_{i<j}n_in_j$ distinct ordinary lines. This completes the proof of Theorem \ref{thm:num_flats_d=2}.
\end{proof}

\section{The complement of the union of two convex sets}
\label{sec:two}

%In this section, we examine the scenario where the coverage configuration by $k$-flats presents the most definitive case. Once $n$, the number of convex sets from Theorem \ref{thm:local}, is exactly 2, we prove that for $S=\Re^d \setminus (K_1 \cup K_2)$ and for every $0 \leq k \leq d$ there is at most one $k$-ordinary flat w.r.t. $S$. Moreover, for each subset $T$ of $\{0,\ldots,d\}$, there exists two convex sets $K_1,K_2 \subset \Re^d$ and a strong cover of $\Re^d \setminus (K_1 \cup K_2)$ that consists of one $k$-ordinary flat for each $k \in T$, and no $k$-ordinary flat for each $k \notin T$.

In this section we present the proof of Theorem~\ref{thm:two_sets-intro}, which fully characterizes the set of possible vectors $(v_0,\ldots,v_d)$ attainable as the numbers of ordinary $k$-flats of $S=\Re^d \setminus (K_1 \cup K_2)$, where $K_1,K_2 \subset \Re^d$ are convex sets. The theorem follows immediately from Claims~\ref{cl:2convexsets}, \ref{cl:K1K2} below.

A construction for $d=2$, in which a strong cover of $S$ consists of exactly one $k$-ordinary flat for every $0 \leq k \leq d$ (or in other words, a construction which corresponds to the vector $(v_0,v_1,v_2)=(1,1,1)$) is presented in Figure \ref{fig:fig4}.  

\begin{figure}[ht]
	\begin{center}
		\scalebox{0.5}{
			\includegraphics[width=0.8\textwidth]{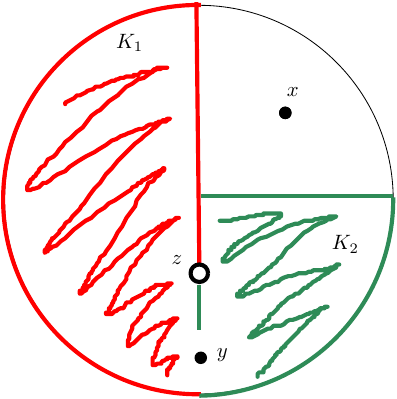}
		}
		\caption{The circle in this figure is centered at the origin. The sets $K_1$ and $K_2$ are colored with red and green, respectively. The set $K_1$ contains the left half circle (including its arc) and the segment $[(0,1),(0,-\frac{1}{3}))$, and $K_2$ contains the quarter right-bottom circle and the segments $[(0,0),(1,0)]$ and $((0,-\frac{1}{3}),(0,-\frac{2}{3})]$. The point $x \in S = \Re^2 \setminus (K_1 \cup K_2)$ is of local dimension $d=2$ and the corresponding 2-ordinary flat is $\Re^2$. The point $y$ is of local dimension $1$ and the corresponding 1-ordinary flat is the $y$-axis. The point $z$ is of local dimension $0$ and the corresponding 0-ordinary flat is $\{z\}$ itself.   }
		\label{fig:fig4}
	\end{center}
\end{figure}

\medskip

The following claim implies that for every $0 \leq k \leq d$, a strong cover of $S$ contains at most one $k$-ordinary flat.

%Formally, we prove the following:
\begin{claim}\label{cl:2convexsets}
	Let $K_1,K_2 \subset \Re^d$ be convex sets, and let $S = \Re^d \setminus (K_1 \cup K_2)$. If $J_1,J_2$ are $k$-ordinary flats w.r.t. $S$ then $J_1 \subset J_2$ or  $J_2 \subset J_1$.
\end{claim} 
\begin{proof}
	Assume on the contrary, that $J_1 \not\subset J_2$ and $J_2 \not\subset J_1$. Then $J_1 \cap J_2$ is a proper subflat of both $J_1$ and $J_2$. Since for $i=1,2$ $J_i$ is an ordinary flat of $S$, there exists $p_i \in S$ and a neighborhood $U_i \subset \Re^d$ of $p_i$ such that $U_i \cap S =U_i \cap J_i$. We can assume that $p_i \in J_i \setminus J_{i+1}$ (where the indices are taken modulo 2), because we can replace $p_i$ by some other $k$-ordinary point in $U_i \setminus J_{i+1}$. This is possible since $\mbox{dim} (J_i \cap J_{i+1}  ) < \mbox{dim} (J_i)$.
	
	Consider the line $\ell = \mbox{aff} (p_1,p_2)$  (see Figure \ref{fig:fig5}). Clearly, $\ell \cap J_i = \{p_i\}$, and since $U_i \cap S =U_i \cap J_i$, $\ell \cap U_i$ contains two points $q_i,q_i'$ very close to $p_i$ from each side of $p_i$ that are not in $J_i$ and therefore not in $S$. It follows that $q_i,q_i'\in K_1 \cup K_2$ and $\ell$ contains (w.l.o.g.) the points $q_2',p_2,q_2,q_1,p_1,q_1'$ in this order, where $p_1,p_2 \in S$ and $q_2',q_2,q_1,q_1' \in K_1 \cup K_2$, in contradiction to the convexity of $K_1$ and $K_2$.    
\end{proof}

%Claim \ref{cl:2convexsets} immediately imlies:

%\medskip

%\noindent \textbf{Theorem \ref{thm:K1K2}.}
%	Let $K_1,K_2 \subset \Re^d$ be convex sets, and let $S = \Re^d \setminus (K_1 \cup K_2)$. Then for every $0 \leq k \leq d$, there exist at most one $k$-ordinary flat. 

%\medskip

%On the other hand, we have complete control on the dimensions $k$ for which $k$-flats exist in the strong cover. Formally, we prove the following:

\begin{claim}\label{cl:K1K2}
	Let $f:\{0,\ldots,d\} \rightarrow \{0,1\}$. Then there exist convex sets $K_1,K_2 \subset \Re^d$ such that the number of $k$-ordinary flats in a strong cover of $S =\Re^d \setminus(K_1 \cup K_2)$ is $f(k)$. 
\end{claim}

\begin{proof}
	We prove the claim by induction on $d$. For $d=1$, consider the following settings:
	\begin{enumerate}
		\item $K_1 \cup K_2 = \Re$, that corresponds to the function $f(0)=0, f(1)=0$.
		\item $K_1 \cup K_2 = \Re \setminus \{a\}$, that corresponds to the function $f(0)=1, f(1)=0$.
		\item $K_1 \cup K_2 = \Re \setminus [a,b]$, that corresponds to the function $f(0)=0, f(1)=1$.
		\item $K_1 \cup K_2 = \{a\} \cup (b,\infty)$ ($b>a$), that corresponds to the function $f(0)=1, f(1)=1$.
	\end{enumerate} 
	For $d>1$, let $g:\{0,\ldots,d-1\} \rightarrow \{0,1\}$ be the restriction of $f$ to $\{0,\ldots,d-1\}$. We abuse notation and identify $\Re^{d-1}$ with the points in $\Re^d$ whose last coordinate is 0. By the induction hypothesis, there exist  $K_1',K_2' \subset \Re^{d-1}$ such that a strong cover of $\Re^{d-1} \setminus(K_1' \cup K_2')$ consists of exactly $g(k)$ ordinary $k$-flats for each $0 \leq k \leq d-1$. 
	
	Let $K_1 = K_1' \cup \Re^d_-$ where $\Re^d_-=\{x=(x_1,\ldots,x_d) \in \Re^d:x_d<0\}$. If $f(d)=1$ then $K_2 = K_2'\cup (\Re^d_+ \cap  \{x=(x_1,\ldots,x_d) \in \Re^d:x_d \leq 1\})$, and if $f(d)=0$ then $K_2 = K_2' \cup \Re^d_+$.
	Indeed, in the first case all the points $(x_1,\ldots,x_d) \in S$ with $x_d > 1$ are $d$-ordinary points, and in the second case there are no $d$-ordinary points in $S$. For all other dimensions we are done by the induction hypothesis. 
\end{proof}

\begin{figure}[ht]
	\begin{center}
		\scalebox{0.7}{
			\includegraphics[width=0.8\textwidth]{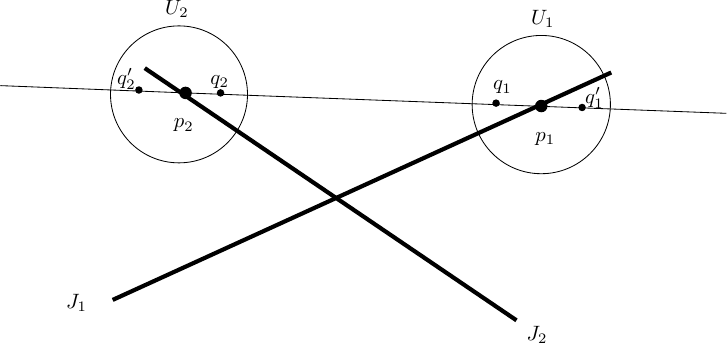}
		}
		\caption{An illustration for the proof of Claim \ref{cl:2convexsets}. }
		\label{fig:fig5}
	\end{center}
\end{figure}

\bibliographystyle{plain}
\bibliography{references}

\begin{thebibliography}{10}

\bibitem{Agarwal2008a}
P.~K. Agarwal, J.~Pach, and M.~Sharir.
\newblock State of the union, of geometric objects: {A} review.
\newblock In {\em Proc. Joint Summer Research Conf. on Discrete and
  Computational Geometry: 20 Years Later, Contemp. Math. 452}, pages 9--48.
  {AMS}, 2008.

\bibitem{AronovCDG17}
B.~Aronov, O.~Cheong, M.~G. Dobbins, and X.~Goaoc.
\newblock The number of holes in the union of translates of a convex set in
  three dimensions.
\newblock {\em Discret. Comput. Geom.}, 57(1):104--124, 2017.

\bibitem{AronovS97CGTA}
B.~Aronov and M.~Sharir.
\newblock The common exterior of convex polygons in the plane.
\newblock {\em Comput. Geom.}, 8:139--149, 1997.

\bibitem{AronovS97SICOMP}
B.~Aronov and M.~Sharir.
\newblock On translational motion planning of a convex polyhedron in 3-space.
\newblock {\em {SIAM} J. Comput.}, 26(6):1785--1803, 1997.

\bibitem{BeiCZ15}
X.~Bei, N.~Chen, and S.~Zhang.
\newblock Solving linear programming with constraints unknown.
\newblock In {\em {ICALP} 2015, Part {I}}, volume 9134 of {\em Lecture Notes in
  Computer Science}, pages 129--142. Springer, 2015.

\bibitem{BjornerK88}
A.~Bj{\"{o}}rner and G.~Kalai.
\newblock An extended {E}uler–{P}oincar{\'{e}} theorem.
\newblock {\em Acta Math.}, 161:279--303, 1988.

\bibitem{CibulkaKKMRV17}
J.~Cibulka, M.~Korbel{\'{a}}{\v{r}}, J.~Kyn{\v{c}}l, V.~M{\'{e}}sz{\'{a}}ros,
  R.~Stola{\v{r}}, and P.~Valtr.
\newblock On three measures of non-convexity.
\newblock {\em Israel J. Math.}, 218:331--369, 2017.

\bibitem{EdelsbrunnerP24}
H.~Edelsbrunner and J.~Pach.
\newblock Maximum {B}etti numbers of {{\v{C}}}ech complexes.
\newblock In {\em SoCG 2024}, volume 293 of {\em LIPIcs}, pages 53:1--53:14.
  Schloss Dagstuhl - Leibniz-Zentrum f{\"{u}}r Informatik, 2024.

\bibitem{Kat77}
G.~O.~H. Katona.
\newblock On a problem of {L}. {F}ejes {T}{\'{o}}th.
\newblock {\em Stud. Sci. Math. Hung.}, 12(1--2):77--80, 1977.

\bibitem{Kov88}
M.~D. Kovalev.
\newblock A property of convex sets and its application.
\newblock {\em Mat. Zametki (in Russian)}, 44:89--99, 1988.

\bibitem{LawrenceM09}
J.~Lawrence and W.~D.~Morris Jr.
\newblock Finite sets as complements of finite unions of convex sets.
\newblock {\em Discret. Comput. Geom.}, 42(2):206--218, 2009.

\bibitem{MatousekV99}
J.~Matou{\v{s}}ek and P.~Valtr.
\newblock On visibility and covering by convex sets.
\newblock {\em Israel J. Math.}, 113(3):341--379, 1999.

\end{thebibliography}

\appendix

\section{Complements of Unions of Disjoint Convex Sets in $\Re^2$}\label{subsec:n-encapsuledR^2}

In this appendix we discuss the number of encapsulated points in the special case where the convex sets are pairwise disjoint. While in $\mathbb{R}^1$, and for three convex sets in $\mathbb{R}^2$, the maximal number of encapsulated points is obtained where the convex sets are disjoint (see Theorem~\ref{thm:3encapsuled}), in the general case the situation is starkly different. 

An easy example presented in~\cite{LawrenceM09} shows that the number of encapsulated points can be as large as $\Omega((\tfrac{n}{d})^d)$: Assuming for simplicity that $d|n$, one can take $\frac{n}{d}-1$ hyperplanes parallel to each of the $d$ axes and take $K_1,\ldots,K_n$ to be the strips between pairs of `consecutive' hyperplanes in the same direction. Clearly, $|S|=(\tfrac{n}{d}-1)^d$ and all its elements are points encapsulated by $K_1,\ldots,K_n$.   

We prove Proposition~\ref{thm:numFlatsR2} which determines the maximal number of encapsulated points in the plane assuming that $\{K_i\}$ are disjoint and $|S|<\infty$, and in particular, asserts that in this case $|S|$ is only linear in $n$ (compared to $\Omega(n^2)$ which can be obtained if disjointness is not assumed). In addition, in Proposition~\ref{Prop:example} we provide an example which shows that in $\Re^d$, as many as $\Omega(n^{\lfloor \frac{d+1}{2} \rfloor})$ encapsulated points can be obtained for disjoint convex sets $K_1,\ldots,K_n$.

%In this appendix we present the proof of Proposition~\ref{thm:numFlatsR2}.

\medskip \noindent \textbf{Proposition~\ref{thm:numFlatsR2} - restatement.}
Let $n \geq 3$ and let $K_1,\ldots,K_n \subset \Re^2$ be pairwise disjoint convex sets, such that $S=\Re^2 \setminus(\cup_{i=1}^n K_i)$ is a finite set of points. Then $|S| \leq 5n-11$. Furthermore, for any $n \geq 3$, this bound is attained for some sets $K_1,\ldots,K_n$.

\begin{remark}
We note that by Theorem~\ref{thm:local}, it is sufficient to assume that $S$ does not contain a segment, since this assumption implies that $S$ is a finite set of points. 
\end{remark}

\begin{proof}[Proof of Proposition~\ref{thm:numFlatsR2}]
	Assume w.l.o.g. that $\forall 1 \leq i \leq n$, $\mbox{dim}(K_i)=2$. This can indeed be assumed, as if some $K_i$ is a segment, a ray or a line, it touches only 2, 1 or 0 points in $S$, hence its contribution to the coefficient of $n$ in the upper bound is at most 2, while we aim at 5. Similarly, if some $K_i$ is a point, then by removing $K_i$ and adding this point to $S$, we just enlarge the ratio between $S$ and $n$. Moreover,  since the assertion of the theorem is trivial for the case of two 2-dimensional sets (in which all smaller-dimension sets lie on a line), we assume from now on that $\forall 1 \leq i \leq n$, $\mbox{dim}(K_i)=2$ and $n \geq 3$.
	
	For $1 \leq i < j \leq n$, let $\ell_{ij} $ be a line that separates $K_{i} $ from $K_{j} $. Let $H_{ij} $ be the open half-plane supported by $\ell_{ij} $ that includes $K_j$. Consider the planar drawing $G$ whose edges are $$\{  \mbox{cl} (K_i) \cap \mbox{cl} (K_j) : 1 \leq i<j \leq n, \dim( \mbox{cl} (K_i) \cap \mbox{cl} (K_j))=1  \},$$ whose faces are $\{Q_i = \bigcap_{j \neq i} H_{ij}\}_{i=1}^n$, and whose vertices are points that belong to the closure of at least 3 $K_i$'s. (See Figure \ref{fig:fig13}.) Note that $\forall 1 \leq i \leq n, \mbox{cl} (Q_i) =  \mbox{cl} (K_i)$.
	
	Denote the number of vertices in $G$ by $v$, the number of edges by $e$, and the number of faces by $f$. By Euler's formula\footnote{The standard formulation of Euler's formula, $v-e+f=2$, refers to a finite planar graph, where the outer region is also considered a face. We can pass between the standard formulation and the formulation discussed in our case, by enclosing all vertices and edges within a large circle, and ignoring the rays extending outside the circle. This transition adds one outer face and adds the same number of edges and vertices.}, $v-e+f=1$.
	\begin{figure}[ht]
		\begin{center}
			\scalebox{0.5}{
				\includegraphics[width=0.8\textwidth]{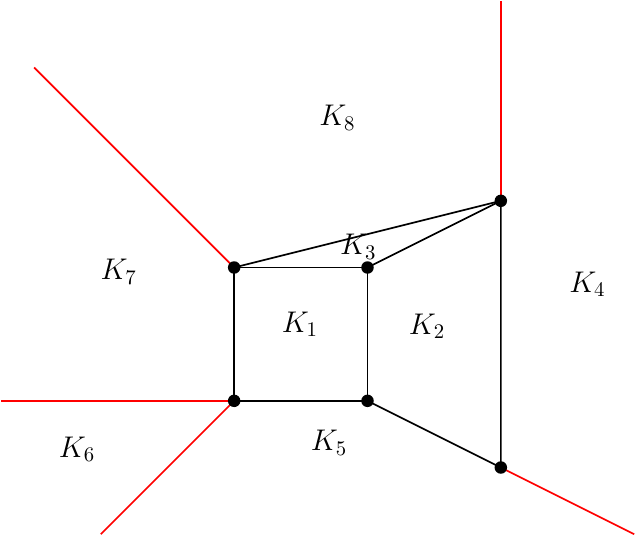}
			}
			\caption{An illustration for the proof of Proposition \ref{thm:numFlatsR2} that presents the planar graph defined by the 8 convex sets $K_1, \ldots, K_8$. Five edges of this graph are rays and are colored red, and 8 edges are segments and are colored black. The 6 vertices are drawn as small discs.}
			\label{fig:fig13}
		\end{center}
	\end{figure}
	Let $X$ be the number of incidences between vertices and edges in $G$. Since each vertex is incident to at least 3 edges, 
	\begin{equation}
		3v \leq X. 
		\label{eq:eq3}
	\end{equation}
	On the other hand, let $e_r$ be the number of rays among the edges of $G$, and let $e_l$ be the number of lines among the edges of $G$. Since each segment is incident with at most 2 vertices, each ray is incident with at most one vertex, and each line is incident with no vertex,  we have 
	\begin{equation}
		X \leq 2e-e_r-2e_l.
		\label{eq:eq1}
	\end{equation}
	\begin{claim}\label{cl:eq}
		\begin{equation}
			e_r+2e_l \geq 3.
			\label{eq:eq2}
		\end{equation}
	\end{claim}
	\begin{proof}[Proof of Claim \ref{cl:eq}]
		Recall the assumption that $\forall 1 \leq i \leq n$, $\mbox{dim}(K_i)=2$ and $n \geq 3$.
		Consider a large disc that contains all the vertices within it. Every edge that extends outside the disc is unbounded.  There must be at least one such edge, since the region outside the circle is not convex. For a similar reason, it cannot be that only one line extends outside the circle, or only two parallel rays, or only two rays in different directions. Therefore, at least 3 rays, or alternatively, at least two lines, extend outside the circle. In any case, it holds that $	e_r+2e_l \geq 3$, as asserted.
	\end{proof}
	
	\noindent Combining together (\ref{eq:eq3}), (\ref{eq:eq1}) and (\ref{eq:eq2}), we have
	\begin{equation}
		3v \leq 2e-3.
		\label{eq:eq4}
	\end{equation}
	Now, combining~\eqref{eq:eq4} together with the equality $v-e+f=1$, we get
	$$\begin{cases}
		e \leq 3f-6  & \\
		v \leq 2f-5, & 
	\end{cases}$$
	and in total, $$v+e \leq 5f-11.$$
	The assertion follows immediately, since the points of $S$ can be only the vertices of $G$ together with at most one point on each edge of $G$.
	
	This proof method inspires the tightness construction illustrated in Figure \ref{fig:fig14}.
\end{proof}
\begin{figure}[htbp]
	\centering
	\begin{subfigure}[b]{0.55\textwidth}
		\centering
		\includegraphics[width=\textwidth]{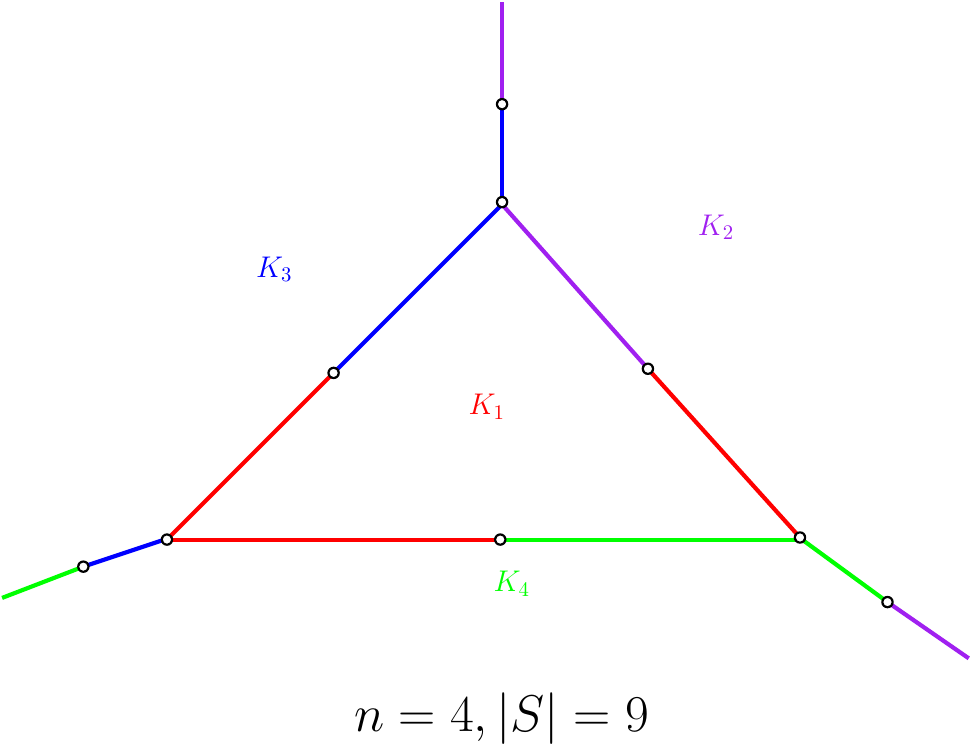}
		\caption{}
		\label{fig:fig14a}
	\end{subfigure}
	\hfill
	\begin{subfigure}[b]{0.65\textwidth}
		\centering
		\includegraphics[width=\textwidth]{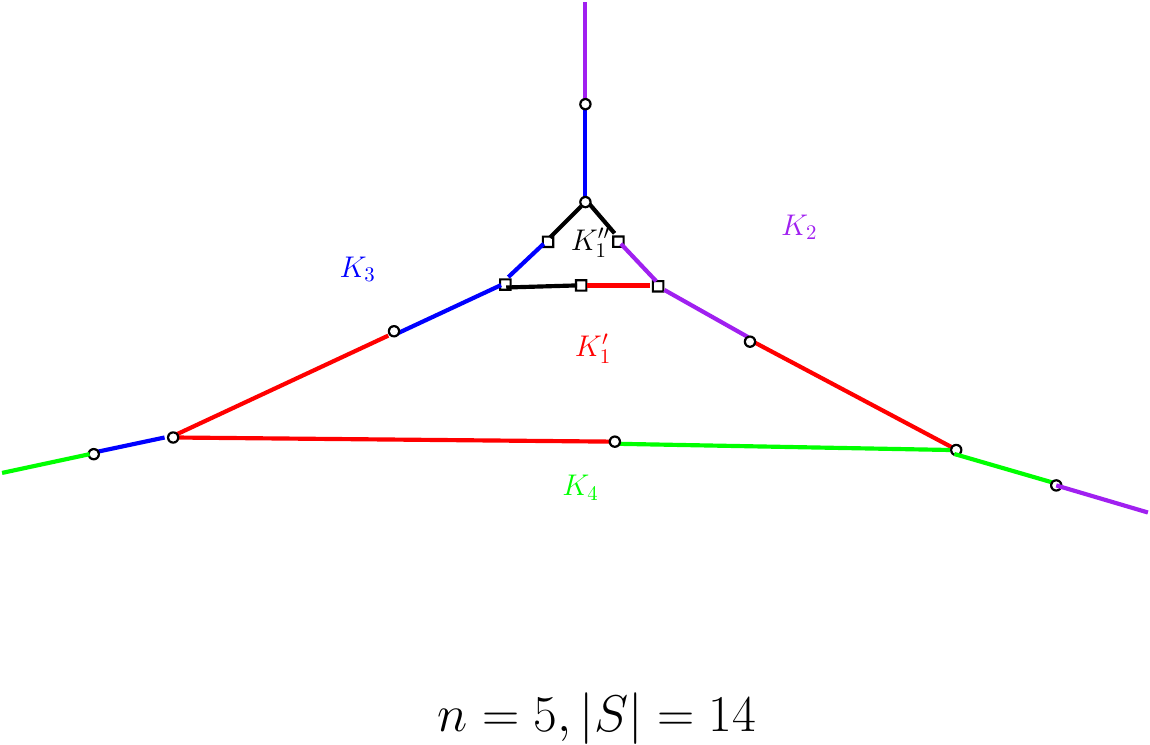}
		\caption{}
		\label{fig:fig14b}
	\end{subfigure}
	\caption{The tightness construction for Proposition \ref{thm:numFlatsR2}. In Fig.~\ref{fig:fig14a}, the construction is presented for $n=4$. Each convex set corresponds to a different color, where an edge is colored with the color of the convex set to which it belongs. The 9 points of $S$ are marked with hollow circles. In Fig.~\ref{fig:fig14b}, the construction is presented for $n=5$, where the set $K_1$ was converted to the sets $K_1'' \subset K_1$, and $K_1'$ which was obtained from the lower part of $K_1$ with a slight modification in the slopes of the non-horizontal edges as can be seen in the figure. In this transition from $n=4$ to $n=5$, five points were added to $S$, which are drawn as squares in the figure.
		Similarly, one can proceed to $n=6,7,\ldots$ where at each step one convex set is added, and five points are added to $S$.} 
	\label{fig:fig14}
\end{figure}

\begin{proposition}\label{Prop:example}
	For any $d,n \in \mathbb{N}$ such that $n>d+1 \geq 4$, there exist pairwise disjoint convex sets $K_1,\ldots,K_n \subset \Re^d$ that encapsulate at least $\Omega(n^{\lfloor \frac{d+1}{2} \rfloor})$ points. 
\end{proposition}

\begin{proof}
	%As for the maximal number of points that can be encapsulated by $n$ pairwise disjoint convex sets in $\Re^d$, the following construction implies a lower bound of $\Omega(n^{\lfloor \frac{d+1}{2} \rfloor})$. 
	Let $m=n^{\lfloor \frac{d+1}{2} \rfloor}$. As in the proof of Theorem \ref{thm:num_flats}, where $n>d+1 \geq 4$, there exists an $m$-neighborly convex $(d+1)$-polytope with $n$ vertices and $\Theta(n^{m}) $ facets. The dual polytope $P$ has $n$ facets and $\Theta(n^{m}) $ vertices. Assume $P$ has a unique ``highest'' vertex $x$ (by ``height'' we mean the $(d+1)$-st coordinate). Choose a point $x' \in \mbox{int} P$ that is higher than all vertices of $P$ except $x$, and apply a central projection from $x'$ to some lower horizontal hyperplane $\pi$ ($\pi=\{z:z_{d+1}=constant\}$).  
		
		% cyclic polytope $\Re^{d+1}$ with $n$ vertices and $\Theta(n^{\lfloor \frac{d+1}{2} \rfloor}) $ facets. Its dual polytope, $P$, has $n$ facets and $\Theta(n^{\lfloor \frac{d+1}{2} \rfloor}) $ vertices.
	
%	Assume that $x$ is the (unique) highest vertex of $P$, and apply a central projection from $x' \in \mbox{int}P$ which is very close to $x$ (above all other vertices of $P$), to some $d$-dimensional horizontal hyperplane $\pi$ that lies strictly below $P$. 
	
	The image of each facet that does not contain $x$ is a $d$-dimensional polytope in $\pi$, and the facets that contain $x$ are projected to unbounded $d$-dimensional polyhedral sets in $\pi$. These images of the facets of $P$ in $\pi$ are $n$ convex sets $\hat{K_1}, \ldots,\hat{K_n}$ with $\hat{K_1}\cup \ldots \cup \hat{K_n}=\pi$.
	
	For each face $F$ in the projection, let $\nu(F)=\min \{i:F \subset \hat{K_i}\}$ (in particular, $\nu(\hat{K_i})=i $) and let $K_i=\bigcup \{F: \nu(F) = i\}$. Each $K_i$ is a convex set, since 
	$$K_i=\hat{K_i} \setminus \bigcup_{j<i}\hat{K_j} = \bigcap_{j=1}^{i-1} (\hat{K_i} \setminus \hat{K_j})=\bigcap_{j=1}^{i-1} (\hat{K_i} \setminus (\hat{K_i} \cap \hat{K_j})),$$ and $\hat{K_i} \cap \hat{K_j}$ ($j<i$) is a face of $\hat{K_i}$ whose removal maintains convexity.
	
	Clearly, $\bigcup_{i=1}^n K_i= \bigcup_{i=1}^n \hat{K_i}=\pi$ .The projections of the vertices of $P$ are extremal points of the $K_i$'s. Therefore we can remove them from the $K_i$'s, to form a complementary set $S$ with $|S|=\Theta(n^{m}) $. 
\end{proof}

\end{document}